\documentclass[12pt]{article}

\usepackage[margin=3cm]{geometry}
\usepackage{amscd}
\usepackage{amsmath}
\usepackage{amssymb}
\usepackage{amsthm}
\usepackage{arydshln}
\usepackage{fancyhdr}
\usepackage{verbatim}
\usepackage[usenames, dvipsnames]{color}
\usepackage{tikz}
\usepackage[all]{xy}
\usepackage[colorlinks=true, linkcolor=blue, citecolor=blue, pagebackref=false]{hyperref}
\usepackage[pagebackref=false]{hyperref}
\usepackage{mathrsfs}
\renewcommand{\mathcal}{\mathscr}
\renewcommand{\epsilon}{\varepsilon}
\renewcommand{\phi}{\varphi}
\newcommand{\by}{\underline {y}}

\newcommand{\bx}{\underline {x}}
\newcommand{\bi}{\underline {i}}

\newcommand{\bj}{\underline {j}}
\newcommand{\bq}{\underline q}
\newcommand{\bz}{\underline z}
\newcommand{\ba}{\underline a}

\newcommand{\dimH}{\text{dim}_{\mathcal H}}

\renewcommand{\dim}{\text{dim}_{\mathcal H}}
\newcommand{\proj}{\textrm{proj}}

\newcommand{\ber}{\nu_{\lambda}}
\newcommand{\onf}{\overline{\nu}_{\lambda}}

\newcommand{\pif}{\pi}

\def\R{{\mathbb R}}

\def\N{{\mathbb N}}

\numberwithin{equation}{section}

\theoremstyle{plain}
\newtheorem{theorem}{Theorem}[section]
\newtheorem{lemma}[theorem]{Lemma}
\newtheorem{remark}[theorem]{Remark}
\newtheorem{definition}[theorem]{Definition}
\newtheorem{corollary}[theorem]{Corollary}
\newtheorem*{theoreminf}{Informal version of the main theorem}

\providecommand{\norm}[1]{\lVert#1\rVert}

\renewcommand{\d}{\,\mathrm{d}}

\title{\textbf{Hausdorff dimension of shrinking targets on Przytycki-Urba\'nski fractals}}
\author{Thomas Jordan \\ email: \href{mailto:thomas.jordan@bristol.ac.uk}{thomas.jordan@bristol.ac.uk}\and 
Henna Koivusalo \\ email: \href{mailto:henna.koivusalo@bristol.ac.uk}{henna.koivusalo@bristol.ac.uk}\\ }
\date{\today}

\begin{document}

\maketitle
\thispagestyle{empty}

%===============================================

\begin{abstract}
Shrinking target problems in the context of iterated function systems have received an increasing amount of interest in the past few years. The classical shrinking target problem concerns points returning infinitely many times to a sequence of shrinking balls. In the iterated function system context, the shrinking balls problem is only well tractable in the case of similarity maps, but the case of affine maps is more elusive due to many geometric-dynamical complications. 

In the current work, we push through these complications and compute the Hausdorff dimension of a set recurring to a shrinking target of geometric balls in some affine iterated function systems. For these results, we have pinpointed a representative class of affine iterated function systems, consisting of a pair of diagonal affine maps, that was introduced by Przytycki and Urba\'{n}ski. The analysis splits into many sub-cases according to the type of the centre point of the targets, and the relative sizes of the targets and the contractions of the maps, illustrating the array of challenges of going beyond affine maps with nice projections. The proofs require heavy machinery from, and expand, the theory of Bernoulli convolutions. 
\end{abstract}

\section{Introduction}

Let $(X, d)$ be a metric space and $T:X\to X$ a Borel measurable transformation. A \emph{shrinking target set} is defined by fixing a point $z\in X$ and a sequence $(r_n)\subset (0, \infty)$ and setting
\begin{eqnarray*}
R^*(z,(r_n))&=&\{x\in X\mid T^n(x)\in B(z,r_n)\text{ for infinitely many }n\in\N\}\\
&=&\limsup_{n\to\infty} T^{-n}(B(z,r_n)).
\end{eqnarray*}
Study of sets with recurrence properties such as this are classical in dynamics. The term shrinking target was coined by Hill and Velani \cite{HillVelani95, HillVelani99}. Since then, the set $R^*$ and its variants have since been studied in a variety of settings, including circle rotations, general dynamical systems under mixing conditions, toral automorphisms, $\beta$-transformations, and the Gauss map (see \cite{Tseng08, Galatolo10, ChernovKleinbock01, BugeaudHarrapKristenseVelani10, ShenWang13, LiWangWuXu14}), and by now the literature is so extensive that it would be impossible to give an exhaustive list. In these works, typical questions to ask about $R^*$  have had to do with its size, for example its measure with respect to some $T$-invariant Borel probability measure on $X$, or, in more geometric set-ups, its Hausdorff dimension. We will focus our attention on the second question for a class of iterated function systems on $\R^d$, which we define next.

Let $f_0,\ldots,f_{k-1}:\R^d\to\R^d$ be contractions, that is, there exist $c_0,\ldots,c_{k-1}\in (0,1)$ such that for all $x,y\in\R^d$, $|f_i(x)-f_i(y)|\leq c_i|x-y|$. The collection $\{f_0, \dots, f_{k-1}\}$ is called an \emph{iterated function system (IFS)}. It is well known that there exists a unique non-empty compact set $F$ such that $F=\cup_{i=0}^{k-1} f_i(F)$ \cite{Hutchinson81}. If we assume the \emph{strong separation condition}, that for all $i\neq j\in \{0,\ldots,k-1\}$ $f_i(F)\cap f_j(F)=\emptyset$, we can define a dynamical system $E:F\to F$ by setting $E(x)=f_i^{-1}(x)$ for each $x\in f_i(F)$. Even in the absence of the strong separation condition, the map $E$ can be well-defined for example by picking a canonical choice of $E=f_i^{-1}$ on any non-empty $f_i(F)\cap f_j(F)\neq \emptyset$. We consider the shrinking target problem for the dynamical system $(F, E)$.

We will begin with a short sketch of the types of arguments applied in the study of dimension of $R^*$ in the IFS setting. The result outlined here can be deduced from \cite{HillVelani95} (for a translation from the results of Hill and Velani to the IFS context, see \cite[Theorem 2.1]{AllenBarany21}, and for related results, also see \cite[Section 4]{Baker24}). For simplicity, assume that the $f_0, \dots, f_{k-1}$ are similarities with contraction ratios $c_0, \dots, c_{k-1}$. Let $F$ be the associated self-similar set. Let $z\in F$ and $\gamma\in (0,1)$. Let $r_n=\gamma^n$ for each $n\in \N$. For the sets
$$R^*(z,(r_n))=\limsup_{n\to\infty} E^{-n}(B(z,r_n))$$
it can be shown that $\dim R^*(z,(r_n))=s(\gamma)$ where 
$$\sum_{i=0}^{k-1} c_i^s=\gamma^{-s(\gamma)}.$$
For general $(r_n)$ the same formula holds with $\gamma=\limsup_{n\to\infty}r_n^{1/n}$. 
The method of proof is to get an upper bound by finding natural coverings of the sets $E^{-n}B(z,r_n)$ and take unions of these covering for all $n$ sufficiently large. For the lower bound, sometimes disguised as a mass transference principle, one has to consider subsets of those points with $x\in E^{-n_m}(R^*((z,(r_{n_m}))))$ along some subsequence $(n_m)$ satisfying $\lim_{m\to\infty} n_{m+1}n_m^{-1}=\infty$. The argument then relies on building a suitable mass distribution on these subsets.

Now consider an IFSs consisting of affine maps, that is, each $f_0, \dots, f_{k-1}$ is of the form $f_i:x\mapsto A_ix+\tau_i$ where $A_i\in \text{GL}_d(\R)$ is a linear map with $\norm{A_i}<1$ and $\tau_i\in\R^d$ is a translation. The same basic methods in the study of Hausdorff dimension of $R^*$ are still used. However, the affine case poses many technical obstacles  of both geometric and dynamical nature. Overcoming these problems usually requires some restrictions to the types of maps being considered. Past work in this context has focused on IFSs with generic translations $\tau_0, \dots, \tau_{k-1}$, and with additional restrictions on the matrices $A_0, \dots, A_{k-1}$ \cite{KoivusaloRamirez18, BaranyTroscheit22}. Perhaps most restrictively, in many past results the returns are not to geometric balls $B(z, r_n)$, but to subsets called \emph{cylinder sets} for which the pre-images under $E$ are easier to characterise. Only limited work has been carried out in the case where the shrinking targets are balls, for self-affine IFSs. Most notably in the case of Bedford-McMullen carpets, a special class of self-affine sets in $\R^2$ based on a fixed $n\times m$ grid, B\'ar\'any and Rams \cite{BaranyRams18} computed the Hausdorff dimension of the shrinking target set (also see \cite{BakerKoivusalo24}). 

All of the existing literature on Hausdorff dimension of the shrinking target set for affine IFSs has one of two main limitations: either, the IFS has to arise from a rigid grid structure; or the class of IFSs can be somewhat broader as long as a generic translations-condition is satisfied, but the returns are to a contrived dynamical subset rather than a ball. In the current work we are able to bridge the gap between these two set-ups, by computing the Hausdorff dimension of the shrinking target set $R^*(z, (r_n))$ in a representative class of IFSs without grid structure. More precisely, we consider shrinking targets on \emph{ PU-sets}, a class of self-affine sets first defined by Przytycki and Urba\'{n}ski \cite{PrzytyckiUrbanski89}. The affine IFS defining the PU-set consists of two affine maps $f_0:(x,y)\to (\lambda x,y/2)$ and $f_2:(x,y)\to (\lambda x+(1-\lambda),y/2+1/2)$, where $\lambda>1/2$. Notice that the two translations $\tau_0, \tau_1$ are specific rather than generic, and that no alignment between images is assumed. We can now informally state our main theorem. The full statement needs more notation and is given as Theorem \ref{thm:main} underneath. 

\begin{theoreminf}
Let $\gamma\in (0,1)$. We compute explicitly the value of $\dim R^*(z, (\gamma^n))$ in the following cases: 
\begin{itemize}
\item[(1)] For all $z\in F$ and for all $\lambda\in (\frac{1}{2}, \frac{1}{2\gamma})$ outside of a set of dimension $0$. 
\item[(2)] For all $z\in F$ satisfying a \emph{unique expansion} condition (see Definition \ref{def:unique}), and for all $\lambda\in (\frac{1}{2\gamma}, 1)$ outside of a set of dimension $0$. 
\item[(3)] For typical choices of $z\in F$ (with respect to a natural measure, see Definition \ref{def:measures}), and Lebesgue almost all $\lambda$ in the \emph{region of transversality} (see Definition \ref{def:region}). 
\end{itemize}
\end{theoreminf}

The PU-family of self-affine IFSs is seemingly simple, but has been one of the main examples driving the intuition in dimension theory of self-affine sets in recent years \cite{BaranyKaenmaki17, BaranyHochmanRapaport19}. We will review some of the relevant ideas in Section \ref{sec:set-up}. 

The PU-fractals have intricate local structure, governed by their projections to the two coordinate directions. In particular, it is intimately entwined with the theory of \emph{Bernoulli convolutions}, the distribution of sums of the form $\sum_{n=1}^\infty \pm \lambda^n$, with the sign chosen randomly. It has been known since the 1930's \cite{KershnerWintner35}, that the resulting Bernoulli convolution measure is always either absolutely continuous or singular, and it was proved early on by Erd\H{o}s \cite{Erdos39} that if $\lambda^{-1}$ is Pisot then the corresponding Bernoulli convolution is singular. The question of which other parameters, if any,  result in singular measures, has been the object of intense attention over the past several decades, and has inspired new techniques in dimension theory, such as the transversality method in the 90's \cite{Solomyak95}. We now know that the set of parameters for which the resulting measure is singular is of dimension zero \cite{Shmerkin}.

The array of methods needed in this article is broad. On the dynamical side of the problem, we extend the theory of shrinking targets for affine IFSs to the geometric target case. On the geometric side, for the dimension analysis, we need to adjust the now-classical transversality techniques \cite{Solomyak95}, as well as apply the most recent results on local properties of Bernoulli convolutions \cite{Shmerkin}. The bulk of the work will be needed for the proof of part 3 of the theorem, see sections \ref{sec:upperbounds} and \ref{sec:case3}. As a corollary to the proofs, Theorem \ref{thm:dynamical}, we also obtain a result for cylinder targets.

\section{Problem set-up and statement of results} \label{sec:set-up}
Let $\{f_0, f_1\}$ be the two-map homogeneous affine iterated function system that contracts by $\lambda$ in the $x$-direction and by $\tfrac 12$ in the $y$-direction. More precisely, let $f_0:(x,y)\to (\lambda x,y/2)$ and $f_2:(x,y)\to (\lambda x+(1-\lambda),y/2+1/2)$. Denote by $F$ the limit set of the iterated function system, which  satisfies $F=f_0(F)\cup f_2(F)$. See Figure 1. 
\begin{figure}[!t]
\includegraphics{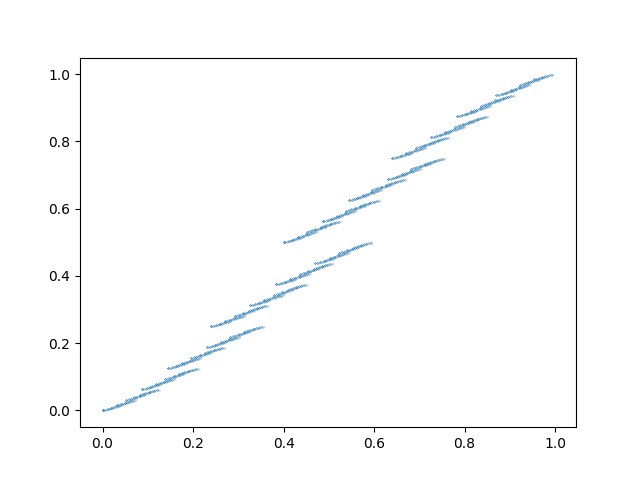}
\caption{An example of a Przytycki-Urba\'{n}ski fractal for $\lambda=0.6$.}
\end{figure}

We will also need to consider another iterated function system, which corresponds to the Bernoulli convolution. Hence, let $g_0, g_1: [0,1]\to [0,1], g_0(x)=\lambda x$ and $g_1(x)=\lambda x+(1-\lambda)$. Notice that in other words, $\{g_0, g_1\}$ is the projection of the system $\{f_0, f_1\}$ to the $x$ coordinate.

Use the notation $\Sigma$ for the sequence space $\{0,1\}^{\mathbb N}$, and $\bi, \bj$ and so on for elements of $\Sigma$ or $\Sigma_n=\{0,1\}^n$. Let $\Sigma^*=\bigcup_{n=1}^\infty \Sigma_n$. Denote by $\sigma$ the shift dynamics on the symbolic space, and also on finite words when appropriate. Furthermore, let $\bi\wedge \bj$ be the maximal joint beginning of $\bi$ and $\bj$, $\bi|_n$  the finite sequence obtained by truncating $\bi$ at the $n$-th symbol, and $|\bi|$ the length of a finite word. Furthermore, let $\bi|_{n}^k$ denote the finite word obtained from the digits of $\bi$ between $n$ and $k$: $i_{n+1}\dots i_k$. Now let
\[
[\bi]=\{\bj\in\Sigma\mid \bj|_{|\bi|}=\bi\}.
\]
Let $f_{\bi}= f_{i_1}\circ \cdots \circ f_{i_{|\bi|}}$  in the usual way. Denote by $\pi=\pi_\lambda$ the natural projection map
\[
\pi:\Sigma\to F, \quad \pi(\bi)=\lim_{n\to \infty} f_{\bi|_n}(0). 
\]
The map $\pi$ is a bijection. Denote by $\pi_I: \Sigma\to [0,1]$ the natural projection map corresponding to the iterated function system $\{g_0, g_1\}$. We will repeatedly use the notation $\bx$ for the point in $\Sigma$ with the property $\pi(\bx)=x\in F$, and analogously for, say, $z\in F$, $\bz\in \Sigma$. We also define the following measures. 

\begin{definition}\label{def:measures}
Denote the Bernoulli measure with equal weights on $\Sigma$ by $\nu$, the push-forward to $F$ by $\onf=\pi_*\nu$, and the push-forward to $[0,1]$ by $\ber=(\pi_I)_*\nu$. Note that, equivalently, $\ber=(\proj_1)_* \onf$, where $\proj_1: [0,1]^2\to [0,1]$ is the projection to the first coordinate.
\end{definition}

Before we look at the shrinking target sets in this setting, let us recall what has been previously shown in terms of the Hausdorff dimension of the set. In Proposition 2 and Theorem 7 of \cite{PrzytyckiUrbanski89} it is shown that  $\dimH F=2+\frac{\log\lambda}{\log 2}$ for those $\lambda\in (1/2, 1)$ for which the dimension of the Bernoulli convolution $\ber$ is $1$. There is  a long history of study of Bernoulli convolutions. Concerning their dimensions some key literature includes: 
\begin{enumerate}
\item
In \cite{Solomyak95} Solomyak showed that $\dim\ber=1$ for Lebesgue almost all $\lambda\in (1/2,1)$, using what is known as the transversality technique.
\item
In \cite{Hochman14} Hochman showed that $\dim\ber=1$ for all $\lambda\in (1/2,1)$ satisfying an exponential separation condition (see \ref{def:expsep}) which holds for all but a set of packing dimension $0$, and his result was refined by Shmerkin \cite{Shmerkin}.
\item
In \cite{Varju2019} Varj\'u showed that $\dim\ber=1$ for all transcendental $\lambda\in (1/2,1)$. 
\end{enumerate}
It follows that combining \cite{Varju2019} with \cite{PrzytyckiUrbanski89} yields that $\dimH F=2+\frac{\log\lambda}{\log 2}$ for all transcendental $\lambda \in (1/2,1)$. However the results of Shmerkin and Solomyak are important for this work as we will be using theorems relying on transversality (Theorem \ref{thm:transversality}) and uniform local dimension estimates under the exponential separation assumption (Theorem \ref{thm:pablo}).  

We now turn to the problem this paper is focused on.
Denote by $E: F\to F$ the expanding map that has $f_i$ as its local inverses, i.e. for any $x\in f_i(F)$, let $E(x)=f_i^{-1}(x)$ -- since $F=f_0(F)\cup f_1(F)$, $E$ is defined at every point. For those $x\in f_0(F)\cap f_1(F)$, we always take $E(x)=f_0^{-1}(x)$. Fix $z\in F$, $0<\gamma<1$, and set the sequence of {\it shrinking targets} to be the sequence of shrinking cubes $(Q_n=Q(z, \gamma^n))\subset F$. We will want to understand the subset of points $R^*\subset F$ which return to the shrinking targets infinitely often, that is, 
\[
R^*(z, \lambda, \gamma)=R^*=\{x\in F\mid E^n(x)\in Q(z, \gamma^n) \textrm{ for infinitely many }n\}. 
\]
We demonstrate in Remark \ref{rem:sharp} that for any sub-exponentially shrinking sequence of targets, the set $R^*$ has the dimension of $F$. Hence, the exponential speed of shrinking is appropriate for dimension problems for $R^*$. 

Before we can precisely state the main theorem, we need to introduce some notation. The formulation of the following definition is from \cite{Shmerkin}. 
\begin{definition}\label{def:expsep}
Let $\mathcal P_n$ be the collection of degree $n$ polynomials with coefficients in $\{0,1,-1\}$. Let 
\[
J=\{\lambda\in (\tfrac 12, 1)\mid \lim_{n\to \infty} \tfrac 1n\log \min_{P\in \mathcal P_n}|P(\lambda)|=-\infty\}. 
\]
It is shown in \cite{Hochman14} that $\dimH J=0$. Denote $I^c=\mathcal E$. We say that $\lambda\in\mathcal E$ satisfy the \emph{exponential separation condition}. The only property of this set that we use is due to Shmerkin \cite{Shmerkin}, and is characterised in Theorem \ref{thm:pablo}. 
\end{definition}

\begin{definition}\label{def:region}
There is a value $\tfrac 12<\bar \lambda<1$ such that for all $\lambda\in (1/2, \bar\lambda)$, 
\[
g(\lambda):=1+\sum_{n=1}^\infty b_n\bar\lambda^n \quad\textrm{ with }b_n\in \{0, -1, 1\}
\]
has no double-zeros (that is, there is no $\lambda$ in this range for which  $g(\lambda)=0$ and $g'(\lambda)=0$). The value of $\bar\lambda$ is approximately $0.668$, see Theorem 2.6 in \cite{ShmerkinSolomyak06}. The interval $(\tfrac 12, \bar \lambda)$ is known as the \emph{region of transversality}. The only property of $\bar\lambda$ that we will need is due to Peres and Solomyak \cite{PeresSolomyak98}  and is characterised in Theorem \ref{thm:transversality}. 
\end{definition}

\begin{definition}\label{def:unique}
Given $\lambda\in(\tfrac 12, 1)$, we say that a point $x\in [0,1]$ has a {\it unique $\lambda$-expansion} if the map $\pi_I:\Sigma\to [0,1]$ is one-to-one at $x$, that is, there a unique $\bi\in \Sigma$ such that $x=\pi_I(\bi)$. 
\end{definition}

Finally, denote the Lebesgue measure on $[0,1]$ by $\mathcal L$.  We are now ready to state the main theorem. 
\begin{theorem}\label{thm:main}
Fix $\gamma\in (0,1)$. Then the following are true: 
\begin{itemize}
\item[(1)] For any $\lambda\in (\tfrac 12, \tfrac {1}{2\gamma})\cap \mathcal E$, and for all $z\in F$, 
\[
\dimH R^*(z, \lambda, \gamma)=\frac{-\log 2}{\log(\lambda\gamma)}. 
\]
\item[(2)] For any $\lambda\in (\tfrac 1{2\gamma}, 1)\cap\mathcal E$, for any $z\in F$ such that $\pi_I(\bz)=\proj_1(z)$ has a unique $\lambda$-expansion, 
\[
\dimH R^*(z, \lambda, \gamma)=2+\frac{\log \lambda}{\log 2} - \frac{\log \gamma}{\log( \gamma\lambda)}. 
\]
\item[(3)] For $\mathcal L$-almost every $\lambda\in [\tfrac 1{2\gamma}, \bar\lambda)$, for $\onf$-almost every $z\in F$, 
\[
\dimH R^*(z, \lambda, \gamma)=\frac{2\log 2+\log \lambda}{\log (2/\gamma)}. 
\]
\end{itemize}
\end{theorem}
The difference between the dimensions in parts (2) and (3) is relating to the different local dimension for the Bernoulli convolution for $\proj_1(z)$. In fact for part (2) the unique expansion assumption could be strengthened to the assumption that the local dimension of the Bernoulli convolution at $\proj_1(z)$ is $-\frac{\log 2}{\log\lambda}$. It may be the case that a general formula in terms of this local dimension exists, but our arguments rely on the local dimensions being either typical or maximal.

We can also exploit Fubini's Theorem and the monotonicity in $\gamma$ of the sets $R^*(z, \lambda, \gamma)$  to obtain the following corollary.
\begin{corollary}
For almost every $\lambda\in (\frac{1}{2},\overline{\lambda})$ we have that for $\onf$ almost all $z\in F$
\begin{enumerate}

\item     
For $\gamma\in (\frac{1}{2\lambda},1)$ we have that
\[
\dimH R^*(z, \lambda, \gamma)=\frac{2\log 2+\log \lambda}{\log (2/\gamma)}. 
\]
\item 
For $\gamma\in (0,\frac{1}{2\lambda})$ we have that
\[
\dimH R^*(z, \lambda, \gamma)=\frac{-\log 2}{\log(\lambda\gamma)}. 
\]
\end{enumerate}
\end{corollary}
\begin{proof}
For $\lambda\in (\frac{1}{2},\overline{\lambda})$ and $\bj\in\Sigma$ we define $f_{\lambda,\bj}:(0,1)\to[0,\infty)$ by
$$f_{\lambda, \bj}(\gamma)=\left\{\begin{array}{lll}\frac{\log 2}{-\log (\lambda\gamma)}&\text{ if }&\gamma\in (0,\frac{1}{2\lambda})\\
\frac{2\log 2+\log \lambda}{\log (2/\gamma)}&\text{ if }&\gamma\in [\frac{1}{2\lambda},1)\end{array}\right.$$
We have by Theorem \ref{thm:main} that for all $\gamma\in (0,1)$
$$\mathcal{L}\times\nu(\{(\lambda,\bj)\in (\frac{1}{2},\overline{\lambda})\times\Sigma:\dimH R^*(\pi(\bj), \lambda, \gamma)<f_{\lambda,\bj}(\gamma)\})=0.$$
Therefore by Fubini for Lebesgue almost every $\lambda$ and $\nu$ almost all $\bj$ we have that
$$\mathcal{L}(\{\gamma\mid\dimH R^*(\pi(\bj), \lambda, \gamma)<f_{\lambda,\bj}(\gamma)\})=0.$$
The result of the corollary now follows by the continuity of $f_{\lambda,\bj}$, the fact that $\onf=\pi_*\nu$ and the fact that the sets $R^*(z, \lambda, \gamma)$ are monotone in $\gamma$.

\end{proof}

In our analysis, we will need to consider the sets $R=\pi^{-1}(R^*(z, \lambda, \gamma))$. Since the shift map is topologically conjugate to $E$ via the projection $\pi$, we have 
\[
R=\{\bi\in \Sigma\mid \sigma^n(\bi)\in \pi^{-1}(Q(z, \gamma^n)) \textrm{ i.o.}\}
\]
We can also consider the shrinking target set defined purely from the perspective of the coding space. We fix $\bz\in \Sigma, \gamma\in (0,1)$ and let $(\ell_n)\subset \N$ be an increasing sequence. We can then consider symbolic shrinking target sets 
\[
R(\bz, (\ell_n)) :=\limsup_{n\to\infty}\{\bi\mid(i_{n+1},\ldots,i_{n+\ell(n)})=(z_1,\ldots,z_{\ell_n})\}. 
\] 
For Hausdorff dimension questions, we can then consider $\pi_{\lambda}(R(\bz, (\ell_n)))$. We will be able to adapt the proof of parts (1) and (2) of Theorem \ref{thm:main} to find the dimensions of some $\pi_{\lambda}(R(\bz, (\ell_n)))$. 

\begin{theorem}\label{thm:dynamical}
Let
\[
\ell_n= \lceil n\frac{\log \lambda}{\log \gamma}\rceil. 
\]
\begin{enumerate}
\item[(1)] 
For all $\lambda\in (\frac{1}{2},\frac{1}{2\gamma})\cap\mathcal{E}$ and for all $\bz\in\Sigma$ we have that
\[
\dimH \pi( R(\bz, (\ell_n)) ) = \frac{-\log 2}{\log(\lambda\gamma)}.
\]
\item[(2)] 
For all $\lambda\in (\frac{1}{2\gamma},1)\cap\mathcal{E}$ we have that

\[
\dimH \pi( R(\bz, (\ell_n)) ) = 2+\frac{\log \lambda}{\log 2}-\frac{\log\gamma}{\log(\gamma\lambda)}. 
\]
\end{enumerate}
\end{theorem}

We complete this section by outlining how the rest of the article is structured. Section 3 contains a collection of results we will need on $\lambda$-expansions and Bernoulli convolutions. Section 4 contains the proofs of all the upper bounds for both Theorem \ref{thm:main} and Theorem \ref{thm:dynamical}. In section 5 we prove the lower bound for parts 1 of both Theorem \ref{thm:main} and Theorem \ref{thm:dynamical}. In section 6 we prove the lower bounds for parts 2 of both Theorem \ref{thm:main} and Theorem \ref{thm:dynamical}. Finally in section 7 we give the proof of the lower bound for part 3 of Theorem \ref{thm:main}.

\section{Preliminaries}\label{sec:prels}

In the following, the notation $f\lesssim g$ means that $f\le Cg$ for some constant $C$. 
\begin{definition}\label{def:alltheells}
Fix $n, r\in \N$ to set the following notation: 
\begin{itemize}
\item $\ell_1(n)\in\N$ to be the smallest integer such that $(\tfrac 12)^{\ell_1(n)}\leq\gamma^n$.
\item $\ell_2(n)\in\N$ to be the smallest integer with $\lambda^{\ell_2(n)}\leq\gamma^n$
\item $k(r)\in\N$ to be the smallest integer with  $\lambda^{k(r)}\leq(\tfrac 12)^r$
\end{itemize}
These will be used repeatedly throughout, with the exception of Section \ref{sec:uniquelower} where the notation is simplified to $\ell_n:=\ell_2(n)$ as the parameter $\ell_1$ is not needed. Further, in Section \ref{sec:case3}, where multiple parameters $\lambda$ are considered, the notation $\ell_2(n, \lambda):=\ell_2(n)$ is used to emphasise the dependence. 
\end{definition}
Note that as $n,r\to \infty$
\[
\frac{\ell_1(n)}{n}\to \frac{\log\gamma}{-\log 2}, \frac{\ell_2(n)}{n}\to \frac{\log\gamma}{\log\lambda}, \textrm{ and }\frac{k(r)}{r}\to \frac{\log 2}{-\log\lambda}. 
\]
Let $z\in F$ and $\gamma\in (0,1)$. Recall that the targets are given by  $Q_n=Q(z, \gamma^n)$. Denote by $R_n\subset \Sigma$ the set of those symbolic points which return at time $n$, that is,
\[
R_n=\{\bi\in \Sigma\mid \sigma^n(\bi)\in \pi^{-1}(Q_n)\}. 
\]
Then notice that $R=\limsup_{n\to \infty} R_n$, so that for all $N$, 
\begin{equation}\label{eq:cover}
R\subset \bigcup_{n\ge N}R_n. 
\end{equation}
With reference to $E:F\to F$, we define
\[
R_n^*=\{x\in F\mid E^n(x)\in Q_n\}\subset F, 
\]
so that for all $N$, 
\begin{equation}\label{eq:cover*}
R^*=\limsup_{n\to \infty} R_n^*\textrm{ and }R^*\subset \bigcup_{n\ge N}R^*_n .  
\end{equation}

The following lemma will be used to connect symbolic sequences in $\Sigma$ to points of $F$. 

\begin{lemma}\label{agreement}
Let $\bi,\bj\in\Sigma$ and $\lambda\in (1/2,1)$. There exists $0<C<1$ such that if $|\bi\wedge\bj|=r$ then $d(\pi(\bi),\pi(\bj))\geq C2^{-r}$.
\end{lemma}
\begin{proof}
Fix $\bi,\bj\in\Sigma$, $\lambda\in (1/2,1)$ and $r\in\N$ such that $|\bi\wedge\bj|=r$. We will assume that $\bi_{r+1}=1$ and $\bj_{r+1}=0$.
Since $\lambda>1/2$ we can find $n\in\N$ such that 
$$(1-\lambda)\lambda+\lambda^{n+1}<\lambda^2-\lambda^{n+1}$$
and we set
$$C_1=\lambda^2-\lambda^{n+1}-((1-\lambda)\lambda+\lambda^{n+1})>0.$$
Note that if
$$\sum_{k=r+1}^{\infty} (\bi_k-\bj_k)2^{k}\geq 2^{-(r+n)}$$
then $d(\pi(\underline{i}),\pi(\underline{j})\geq 2^{-(r+n)}$. So we now suppose that 
$$\sum_{k=r+1}^{\infty} (\bi_k-\bj_k)2^{k}\leq 2^{-(r+n)}.$$
In this case we have that $(\bi_{r+1},\ldots,\bi_{r+n})=(1,0,\ldots,0)$ and $(\bj_{r+1},\ldots,\bj_{r+n})=(0,1,\ldots,1)$.
Thus
\begin{eqnarray*}
|\pi(\bi)-\pi(\bj)|&\geq& \lambda^{r+2}-\lambda^{r+n+1}-(1-\lambda)\lambda^{r+1}-\lambda^{r+n+1}\\
&=&\lambda^r(\lambda^2-\lambda^{n+1}-((1-\lambda)\lambda+\lambda^{n+1})\\
&\geq&C_1\lambda^r\geq C_12^{-r}
\end{eqnarray*}
Thus if we let $C=\min \{2^{-n},C_1\}$ we have
$$d(\pi(\underline{i}),\pi(\underline{j})\geq C2^{-r}.$$
\end{proof}

\begin{remark}\label{rem:structure}
Geometrically, the points in $R^*_n$ consist of the $2^n$ pre-images of $Q_n$ under $E$, each of which is contained a rectangle of sides $(\tfrac 12)^n\gamma^n$ and $\lambda^n\gamma^n$.

The above Lemma \ref{agreement} implies that if for some $x=\pi(\bx), y=\pi(\by)\in F$, $r\in \N$, we have $y\in Q(x, C2^{-r})$, then $|\by\wedge \bx|\ge r$. It follows that, up to constants, the structure of any point $\bi \in R_n$ is such that there is a free choice of the first $n$ symbols, followed by $\ell_1(n)$ symbols which agree with the centre $\bz$, followed by an ambiguous region, where $\sigma^n(\bi)$ has to map under $\pi_I$ close enough to $\pi_I(\bz)$ but it does not have to agree with it, i.e. we need to know that 
\begin{equation}\label{eq:closetox}
|\pi_I(\sigma^n(\bi) ) - \pi_I(\bz)|\le \gamma^n, 
\end{equation}
but there are (usually) more than one such $\bi$. 

Finally, notice that if $\bi$ is such that $\pi [\sigma^n(\bi)|_{\ell_2}]\cap Q(z, \gamma^n)\neq \emptyset$, then every $\bj\in [\sigma^n(\bi)|_{\ell_2}]$satisfies $\pi(\bj)\in Q(z, 2\gamma^n)$. Hence being an element of $R_n$ only affects the first $\ell_2(n)+n$ digits (up to constants). 
\end{remark}

The following Theorem is due to Shmerkin and will be the form in which the exponential separation condition of Definition \ref{def:expsep} is applied in practice. 

\begin{theorem}[Theorem 1.5 in \cite{Shmerkin}]\label{thm:pablo}
Choose $\lambda\in \mathcal E$, and let $\epsilon>0$. There is  $C_1>0$ such that for all $x\in [0,1]$ and $r>0$, 
\[
{\bar \nu}_{\lambda}(B(x,r))\le C_1r^{1-\epsilon}. 
\]

\end{theorem}

Theorem \ref{thm:pablo} has the following straightforward corollary on the number of sequences $\bi\in \Sigma_k$ that can get close to a given point $x\in [0,1]$ when projected by $\pi_I$.  This result is utilised many times over the course of the proofs, as is implied by Remark \ref{rem:structure}. 

\begin{corollary}\label{cor:branching}
Let $\lambda\in\mathcal{E}\cap (1/2,1)$,  $\rho>0$ and $\epsilon\in (0,1)$. There is a constant $C_1>0$ such that for all $x\in [0,1]$ and $k\in\N$
\[
N_{k}(x)=\#\{\bi \in \Sigma_{k}\mid \pi_I[\bi]\cap [x-\rho\lambda^k,x+\rho\lambda^k]\neq \emptyset\}\le C_12^k\lambda^{k(1-\epsilon)}
\] 
\end{corollary}
\begin{proof}
Notice that for any $|\bi|=k$, if $\pi_I[\bi]\cap  [x-\rho\lambda^k,x+\rho\lambda^k]\neq\emptyset$, then 
\[
\pi_I[\bi]\subset [x-(\rho+1)\lambda^k,x+(\rho+1)\lambda^k]
\]
Thus
$$\overline{\nu}_{\lambda} ([x-(\rho+1)\lambda^k,x+(\rho+1)\lambda^k])\geq N_{k}(x)2^{-k}$$ 
By applying Theorem \ref{thm:pablo} we get that  there exists a $C_1>0$ such that
$$\overline{\nu}_\lambda([x-(\rho+1)\lambda^k,x+(\rho+1)\lambda^k])\leq C_1(\lambda^{k(1-\epsilon)}.$$
Thus
$$
N_k(x)\leq C_1(\lambda^{k(1-\epsilon)}2^{k}=C_12^k\lambda^{(1-\epsilon)}.
$$
\end{proof}

In the case of $z$ with a unique expansion, we can of course expect a better bound. 
\begin{lemma}\label{uelocaldim}
Let $\bj$ be the unique $\lambda$-expansion of some $\overline{\pi}_{\lambda}(\bj)$. We have that
$$\lim_{k\to\infty} k^{-1}\log N_k(\overline{\pi_\lambda}(\bj))=0.$$
\end{lemma}
\begin{proof}
This is an immediate consequence of Corollary 8.5.5 in B\'{a}r\'{a}ny, Simon and Solomyak \cite{BaranySimonSolomyak23}. 
\end{proof}

As well as the exponential separation condition, in Section \ref{sec:case3} we will need to make use of transversality arguments. The transversality property of the parameterised family we consider can be expressed in many ways, but for our purposes a convenient form is given as \cite[Lemma 2.3]{PeresSolomyak98} which we quote underneath.  Recall Definition \ref{def:region} for the notation $\bar\lambda$. 
\begin{theorem}[From \cite{PeresSolomyak98}]\label{thm:transversality}
Fix any $[\lambda_0, \lambda_1]\subset (\tfrac12, \bar \lambda)$. Then, there is a constant $C>0$ such that for all $\rho>0$ and all $\bi, \bj\in \Sigma$ such that $\bi_1\neq \bj_1$, 
\[
\mathcal L\{\lambda\in [\lambda_0, \lambda_1]\mid |\sum_{k=1}^\infty (i_k-j_k)\lambda^k|<\rho\}\le C\rho. 
\]
\end{theorem}

Sometimes we need to choose $\lambda$ and the centre $z$ in such a way that not only is $\bar\nu_\lambda(B(\pi_I(\bz), r)\ge r^{1+\epsilon}$ but also, so is $\bar \nu_\lambda(B(\pi_I(\sigma^k(\bz)), r)\ge r^{1+\epsilon}$. The following lemma will be applied to show that for a full measure set of $\lambda$, this happens in a set of full $\bar \nu_\lambda$-measure. 

\begin{lemma}\label{lem:msrballlower}
Fix $\epsilon,\xi>0$ and $[\lambda_0,\lambda_1]\subset [1/2,1)$. There is a set $A\subset [\lambda_0,\lambda_1]$ with $\mathcal{L}(A)\geq\lambda_1-\lambda_0-\epsilon$ such that for any $\lambda \in A$, for $\nu$-a.e. $\bi$, for all $n$ sufficiently large the measure
\[
\bar\nu_\lambda(B(\pi(\sigma^n(\bi)), \lambda^{\xi n})\ge \lambda^{\xi n(1+\epsilon)}.  
\]
\end{lemma}
\begin{proof}
Let $\lambda\in[\lambda_0,\lambda_1]$, and define: 
\[
A_n=\{x\in[0,1]\mid \bar\nu_\lambda(B(x, \lambda^{\xi n}))\le \lambda^{\xi n(1+\epsilon)}\}. 
\]
Subdivide the unit interval into $\lambda^{-\xi n}+1$ disjoint subintervals of size $\lambda^{-\xi n}$ -- if such an interval contains any point of $A_n$, its $\bar\nu_\lambda$-measure is $\le \lambda^{\xi n(1+\epsilon)}$. This means that 
\[
\bar\nu_\lambda(A_n)\le (\lambda^{-\xi n})\lambda^{\xi n(1+\epsilon)}\le \lambda_1^{\xi n\epsilon,}
\]
 Now define
\[
\pi_{I}^{-1}(A_n)=: B_n. 
\]
Then by the definition $\bar\nu_\lambda=(\pi_I)_*\nu$ and invariance of $\nu$ with respect to $\sigma$, we have $\nu(\sigma^{-n}(B_n))\le \lambda_1^{\xi n\epsilon}$. Set $f(\lambda, \omega)=\sup\{n\mid \sigma^n(\omega)\in B_n\}$. For a fixed $\lambda$, by the bound on the measure of $B_n$, there is a constant $C>0$ uniform in $\lambda$ (based on $\lambda_1$) such that 
\[
\int f(\lambda, \omega)\d\nu(\omega)\le C. 
\]
This altogether means that 
\[
\int_{\lambda_0}^{\lambda_1}\int f(\lambda, \omega)\d\nu(\omega)\d\mathcal L(\lambda)<C
\]
and by Fubini
\[
\int\int_{\lambda_0}^{\lambda_1} f(\lambda, \omega)\d\mathcal L(\lambda)\d\nu(\omega)<C.
\]
Thus for $\nu$ almost all $\omega$ we have that $\int f(\lambda, \omega)\d\mathcal L(\lambda)<\infty$. Thus for $\nu$ almost all $\omega$ we can find $n\in\N$ such that
$$\mathcal{L}(\{\lambda\in [\lambda_0,\lambda_1]:f(\lambda,\omega)>n\})\leq\epsilon$$
and the result follows by taking
$$A=[\lambda_0,\lambda_1]\backslash \{\lambda\in [\lambda_0,\lambda_1]\:f(\lambda,\omega)>n\}.$$

\end{proof}

The following remark connects two forms in which the dimension value corresponding to case (2) in Theorem \ref{thm:main} appears. 

\begin{remark}\label{rem:value}
In the computations, the dimension value 
\[
t(\gamma):=2+\frac{\log \lambda}{\log 2}-\frac{\log\gamma}{\log(\gamma\lambda)}
\]
will often show up in the form 
\begin{align*}
&\frac{\log 2 (2+\frac{\log \gamma}{\log \lambda})+\log(\gamma\lambda)}{\log 2(1+\frac{\log\gamma}{\log\lambda})}\\
&\quad= \frac{2}{1+\tfrac {\log \gamma}{\log \lambda}}+\frac{\log \gamma}{\log (\gamma \lambda)}+\frac{\log \gamma}{\log 2(1+\tfrac{\log \gamma}{\log \lambda})}+\frac{\log \lambda}{\log 2(1+\tfrac {\log \gamma}{\log \lambda})}\\
&\quad =\frac{1}{1+\tfrac {\log\gamma}{\log \lambda}}(2+\frac{\log \lambda}{\log 2})+\frac{\log\gamma}{\log(\gamma\lambda)}(1+\frac{\log \lambda}{\log 2})\\
&\quad=\frac{1}{1+\tfrac {\log\gamma}{\log \lambda}}(2+\frac{\log \lambda}{\log 2})+\frac{\log\gamma}{\log(\gamma\lambda)}(2+\frac{\log \lambda}{\log 2})-\frac{\log \gamma}{\log (\gamma\lambda)}\\
&\quad =t(\gamma).
\end{align*}
In this computation we made use of the fact that 
\[
\frac{1}{1+\tfrac {\log \gamma}{\log \lambda}}+\frac{\log \gamma}{\log (\lambda\gamma)}=1. 
\]
\end{remark}

Finally, we record the following standard result that we will use to give lower bounds for the Hausdorff dimension. 
\begin{lemma}\label{massdistribution}
Let $A\subset\R^d$ be a Borel set, and let $\mu$ be a Borel probability measure with $\mu(A)=1$. We have the following: 
\begin{enumerate}
\item[(1)]
If $\liminf_{r\to 0}\frac{\log\mu(B(Q(x,r)))}{\log r}\geq s$ for all $x\in A$ then $\dimH A\geq s$.
\item[(2)]
If $\iint\frac{\d\mu(x)\d\mu(y)}{|x-y|^s}<\infty$ then $\dimH A\geq s$.
\end{enumerate}
\end{lemma}
This result follows from Theorems 3.3.14 and 3.2.4 in \cite{Edgar98}. The nature of the measure we will be constructing means we will need a limit inferior as in (1) above, rather than the same result with a limit, that is perhaps more common elsewhere.

\section{Upper bounds}\label{sec:upperbounds}
In this section we prove the upper bounds for both Theorems \ref{thm:main} and \ref{thm:dynamical}. We use the standard approach briefly described in the introduction. That is, we cover the shrinking target set by first covering each time $n$ return, $E^{-n}(Q_n)$, by cubes of side length $2^{-r}$, and then take the union of these covers for $n\geq N$; compare to  \eqref{eq:cover*}. 

For a given $n$, there are three strategies for covering the pre-images $E^{-n}(Q_n)$ that are relevant to the upper bounds at the dimension values of Theorem \ref{thm:main}. The very first strategy, which gives the dimension value of (1), is to use a single cube of side $2^{-r_0}$, where $r_0$ is chosen such that $k(r_0)$ is approximately $n+\ell_2(n)$ (for notation, recall Definition \ref{def:alltheells}). The second strategy gives the dimension value of (3). Here, we cover by cubes of side-length $2^{-r}$ where $r=n+\ell_1(n)$. The final method is to cover by cubes of side $2^{-r}$ where $r=n+\ell_{2}$. This gives the dimension value from case (2) in Theorem \ref{thm:main}. 
The first two strategies give upper bounds which hold for all centres $z$, whereas the third method only applies in the case when $z$ satisfies the unique expansion condition, as it makes use of the fact that  there is (almost) no branching between times $n+\ell_1(n)$ and $n+\ell_2(n)$. At the end of the first and third case we will explain how they can be adapted to give the upper bounds in Theorem \ref{thm:dynamical}.

\subsection{The case \texorpdfstring{$r=r_0$}{r=r0}: } 

We fix $\gamma,\lambda$ and $\bj$ and cover the set $R^{*}=R^{*}(\pi(\bj),\lambda, \gamma)$. Choose $r_0$ large as possible so that we can use a single cube of side $(\tfrac 12)^{r_0}$ to cover each of the $2^n$ pre-images $E^{-n}(Q_n)$, which are rectangles with the side-lengths  $\lambda^n\gamma^n$ and $(\tfrac 12 \gamma)^n$. This leads to $r_0$ the largest integer with $2^{-r_0}\ge \lambda^{n+\ell_2(n)}$. Hence, we obtain the upper bound 
\[
\mathcal H^s(R^*)\lesssim \sum_{n\ge N} 2^n(\lambda^n\gamma^n)^s, 
\]
which sum converges, and hence gives an upper bound to the dimension, when 
\[
s>\frac{-\log 2}{\log(\lambda\gamma)}. 
\]
This gives the upper bound for part 1 of Theorem \ref{thm:main}. 

For Theorem \ref{thm:dynamical} when $\ell_n=\lceil n\frac{\log\lambda}{\log \gamma}\rceil$ we have that if $\pi(\bj)=z$ and $\ell_n=\lceil n\frac{\log\lambda}{\log\gamma}\rceil$ then 
$$\pi_{\lambda}(\{\bj\in\Sigma\mid\sigma^n(i)\in[\bj|_{l_n}]\})\subseteq Q_n(z,\gamma^n).$$
Thus
$\pi_{\lambda}(R(\bj,(\ell_n)))\subseteq R^{*}(\pi(\bj),\lambda, \gamma)$ and so it immediately follows that
$$\dimH R(\bj,(\ell_n))\leq \frac{-\log 2}{\log(\lambda\gamma)}.$$

\subsection{The case \texorpdfstring{$r=n+\ell_1(n)$}{r-n+l1(n)}: }

In this case we cover with cubes of side length $2^{-r}=2^{-(n+\ell_1(n))}$ which, by choice of $\ell_1(n)$, corresponds to covering by cubes of side length $2^{-n}\gamma^n$, up to constants. We note as in the previous case, that each of the sets $E^{-n}(Q_n)$ is contained in a rectangle of sides $(\tfrac 12)^n\gamma^n$ and $\lambda^n\gamma^n$. We can cover these by cubes of side-length $\gamma^n(\tfrac 12)^n$. For each of the pre-images, there are (up to constants)
\[
\frac{\lambda^n\gamma^n}{\gamma^n(\tfrac 12)^n}=\frac{\lambda^n}{(\tfrac 12)^n}
\]
of them. This means that 
\[
\mathcal H^s(R^*)\lesssim \sum_{n\ge N} 2^n\frac{\lambda^n}{(\tfrac 12)^n}((\tfrac 12)^n\gamma^n)^s, 
\]
which sum converges when 
\[
s>\frac{-2\log 2-\log \lambda}{\log (\gamma\cdot\tfrac 12)}. 
\]
This gives the upper bound for part 3 of Theorem \ref{thm:main}.

\subsection{The unique expansion case with \texorpdfstring{$r=n+\ell_2(n)$}{r=n+l2(n)}:}\label{subsec:uppercase2}
In this case we aim to cover with cubes of side length $2^{-r}=2^{-(n+\ell_2(n))}$. Let $z=\pi_{\lambda}(\bz)$ satisfy the unique expansion condition, that is, $\pi_I(\bz)$ has unique $\lambda$-expansion $\bz$. We then have that $\pi_{\lambda}(\bi)\in E^{-n}(Q_n)$ implies $|\pi_I(\sigma^n(\bi))-\pi_I(\bz)|\leq\gamma^n$. Recall that
 \[
N_k(z)=\#\{\bi\in \Sigma_k\mid \pi[\bi]\cap Q(z, \gamma^k)\neq \emptyset\}.
\]
 Thus, since $\ell_2(n)$ is the smallest integer satisfying $\lambda^{\ell_2(n)}\le \gamma^n$, there are at most $N_{\ell_2(n)}$ choices for $(i_{n+1},\ldots,i_{n+\ell_2(n)})$. Thus $E^{-n}(Q_n)$ can be covered by $2^nN_{\ell_2(n)}$ cylinders  from level $\ell_2+n$, Furthermore,  each of those can be covered by $\lambda^{n+\ell_2(n)}2^{n+\ell_2(n)}$ squares of side $2^{-(n+\ell_2(n))}$, up to constants.

Therefore we have that for $s\geq 0$, the Hausdorff measure of $R^*(z, \lambda)$ can be estimated 
$$\mathcal H^s(R^*(z,\lambda))\lesssim \lim_{m\to\infty}\sum_{n\geq M}2^nN_{\ell_2(n)}\lambda^{n+\ell_2(n)}2^{n+l_2(n)}2^{-s(n+\ell_2(n))}.$$
By Lemma \ref{uelocaldim} $\lim_{k\to\infty} k^{-1}\log N_k(\bz)=0$. Hence, we have that $\mathcal H^s((R^*(z,\lambda))<\infty$ if 
\[
2^{2+\log(\gamma)/\log(\lambda)-s(1+\log(\gamma)/\log(\lambda))}\lambda\gamma<1
\]
and thus when
$$s>\frac{\log(2)(2+\frac{\log(\gamma)}{\log(\lambda)})+\log(\gamma\lambda)}{\log(2)(1+\frac{\log(\gamma)}{\log(\lambda)})}.$$
In particular
$$\dimH(R^*(z,\lambda))\leq\frac{\log(2)(2+\frac{\log(\gamma)}{\log(\lambda)})+\log(\gamma\lambda)}{\log(2)(1+\frac{\log(\gamma)}{\log(\lambda)})}$$
which is the upper bound in part 2 of Theorem \ref{thm:main} by Remark \ref{rem:value}. 

For the dynamical shrinking target result Theorem \ref{thm:dynamical}, we have that for any $\bj\in\Sigma$,
$$\pi(R(\bj,(\ell_n)))=\limsup_{n\to\infty}\pi(\sigma^{-n}([\bj|_{\ell_n}]).$$
Similarly to above it is true that $\pi(\sigma^{-n}([\bj|_{\ell_n}])$ can be covered by $\lceil 2^{2n+\ell_n}\lambda^{n+\ell_n}\rceil$ boxes of diameter $2^{-n-\ell_n}$ and so
$$\mathcal H^s(\pi(R(\bj,\lambda))\leq  \lim_{M\to\infty}\sum_{n\geq M}\lambda^{n+\ell_n}2^{2n+\ell_n}2^{-s(n+\ell_n)}.$$
Thus we get 
$$\dimH R(\bj,(\ell_n))\leq \frac{\log(2)(2+\frac{\log(\gamma)}{\log(\lambda)})+\log(\gamma\lambda)}{\log(2)(1+\frac{\log(\gamma)}{\log(\lambda)})}$$
and the upper bound for part 2 of Theorem \ref{thm:dynamical} follows by Remark \ref{rem:value}.

\section{Lower bound in Theorem \ref{thm:main} (1): dim \texorpdfstring{$\leq 1$}{leq 1}}

We split the proof of the lower bound into three parts, corresponding to the three parts of Theorem \ref{thm:main}. The first part (part (1) of Theorem \ref{thm:main}) is what we consider in this section. In this case,  $\lambda\in (1/2,(2\gamma)^{-1}]$, and the dimension value is expected to be $-\log 2/\log(\lambda\gamma)$. The argument in this section holds for all those $\lambda$ in this range which satisfy the exponential separation condition of Definition \ref{def:expsep}, and is independent of the centre for the shrinking target $z$.

Fix $\gamma \in (0,1)$ and $\lambda\in (1/2,(2\gamma)^{-1}])$ which satisfies the exponential separation condition. As the bound in part (1) does not depend on $z$, let us fix an arbitrary $z\in F$ with $\underline{z}\in \Sigma$ satisfying $\pif(\underline{z})=z$. Our aim is to construct a subset of $R^*(z,\lambda)$ with sufficiently large dimension.  We start by introducing a strictly increasing sequence  of positive integers, $(n_m)_{m\in\N}$ such that $\lim_{m\to\infty} \frac{n_m}{n_{m+1}}=0$ and $n_{m+1}-n_m>\lceil n_m\frac{\log \gamma}{\log \lambda}\rceil$ for every $m\in \N$. We let
$$\mathcal{C}=\bigcap_{m=1}^{\infty}R_{n_m}^*\subset R^*(z, \lambda, \gamma)$$
and our aim is to show that $\dim\mathcal C\geq -\frac{\log 2}{\log(\lambda\gamma)}$. To achieve this we construct a suitable measure $\bar \mu$ on $\mathcal C$, and apply the mass distribution principle. This measure $\bar \mu$ is constructed by projecting an appropriate measure $\mu$ from the shift space onto $\mathcal C$. 

Recall that for each $m\in\N$,  $\ell_2(n_m)$ is the smallest integer for which $\lambda^{\ell_2(n_m)}<\gamma^n$. Thus, if for some $\underline{x}\in\Sigma$ we have $(x_1,\ldots,x_{\ell_2(n_m)})=(z_1,\ldots,z_{\ell_2(n_m)})$ then $\pi(\bx)\in Q_n(z).$ 
Now for each $m\in\N$ we define $L_{m}=\sum_{i=0}^m \ell_2(n_{i})$ with $n_0=0$. We would like to define $\mu$ in such a way that it is supported on $\cap_{m=1}^\infty R_{n_m}$, that is, it is concentrated on the digits of $\bz$ between levels $n_m$ and $n_m+\ell_2(n_m)$, and between levels $n_m+\ell_2(n_m)$ and $n_{m+1}$ it is as uniform as possible. Hence, we set for a finite word $[\underline{x}]$, the following: If there exists $m\in\N$ where $x_{n_m+i}\neq z_i$, $0<i\leq \ell_2(n_m)$ and $n_m+i\leq |\underline{x}|$ then $\mu([\underline{x}])=0$. On the other hand if $x_{n_m+i}=z_i$ for all $m\in\N$ and $0<i\leq\ell_2(n_m)$ where $n_m+i\leq |\underline{x}|$ we set; 

$$\mu([\underline{x}])=
\left\{\begin{array}{lll}
2^{L_m-|\underline{x}|}&\text{ if }&n_m+\ell_2(n_m)<|\underline{x}|\leq n_{m+1}.\\
2^{L_{m-1}-n_m}&\text{ if }& n_m<\underline{x}\leq n_m+\ell_2(n_m)
\end{array}\right.
$$
This means that the measure $\mu$ is defined for all cylinder sets and
$$\sum_{|\underline{x}|=j}\mu([x_1,\ldots,x_j])=1$$
for all $j\in\N$ and so $\mu$ extends to a Borel probability measure on $\Sigma$ by the Carath\'eodory's extension theorem. We set $\overline{\mu}=\pi_{*}\mu$ and note that $\overline{\mu}(\mathcal{C})=1$. This second statement follows as if $\underline{x}$ satisfies $\pi(\underline{x})\notin \mathcal{C}$ then there exists $m\in\N$ where $x_{n_m+i}\neq z_i$ and $0<i\leq \ell_2(n_m)$ thus for $|\underline{x}|\geq n_m+i$ we have  $\mu([\underline{x}])=0$.

For the dimension lower bound our aim is to use part 1 of Lemma \ref{massdistribution} with the measure $\overline{\mu}$. We fix $\lambda\in (1/2,1)\cap\mathcal{E}$ and $C$ as in Lemma \ref{agreement}. Recall the notation for $\ell_1(n), \ell_2(n), k(r)$ from Definition \ref{def:alltheells}. We fix $R>0$ such that $R$, an arbitrary $x\in \mathcal{C}$ and consider cubes $Q(x,R)$. We consider $m\in \N$ such that $R\in [C2^{-r_{-1}(n_{m+1})},C2^{-r_{-1}(n_{m})}]$,  where 
\begin{itemize}
\item $r_{-1}=r_{-1}(n_m)$ such that $k(r_{-1}(n_m))=n_m$, and 
\item $r_0=r_0(n_m)$ such that $k(r_0(n_m))=n_m+\ell_2(n_m)$. 
\end{itemize}
Note that since $\lambda<\frac{1}{2\gamma}$ we have $r_0>n_m$. Underneath, we will handle separately three ranges of values $R\in [C2^{-r_{-1}(n_{m+1})},C2^{-r_{-1}(n_{m})}]$, essentially, $R\in [C2^{-n_m},C2^{-r_{-1}(n_m)}]$, $R\in [C2^{-r_0},C2^{-n_m}]$ and $R\in[C2^{-r_{-1}(n_{m+1})}, C2^{-r_0}]$. 
\begin{lemma}\label{boundby1}
Let $x\in \mathcal{C}$. Let $r\in \N$ satisfy $r_{-1}(n_m)\leq r\leq n_m$. Then for all $R\in[C2^{-(r+1)},C2^{-r})$ we have that
$$\overline{\mu}(Q(x,R))\leq 2^{L(m-1)-r}.$$
\end{lemma}
\begin{proof}
Let $\underline{x}\in\Sigma$ satisfy that $\pi(\underline{x})=x$. Suppose that $\underline{y}\in\Sigma$ with $$|\pif(\underline{y})-\pif(\underline{x})|<R.$$
We have that by Lemma \ref{agreement} for $1\leq i\leq r$, $\bx_i=\by_i$. Thus using the definition of $\mu$ gives
\begin{eqnarray*}
\overline{\mu}(Q(x,R))&=&\mu(\{\by\mid|\pi(\underline{x})-\pi(\by)|<R\})\\
&\leq&\mu(\{\underline{y}\mid x_i=y_i\text{ for }1\leq i\leq r\})\\
&\leq&2^{-r+L(m-1)}
\end{eqnarray*}  
and the result follows.
\end{proof}
\begin{lemma}\label{r_0bound}
We have that for $x\in \mathcal{C}$, $n_m\leq r\leq r_0(n_m)$ and $R\in[C2^{-(r+1)},C2^{-r})$ that
$$\overline{\mu}(Q(x,R))\leq 2^{-n_m+L_{m-1}}.$$
\end{lemma}
\begin{proof}
Let $\underline{x}\in\Sigma$ satisfy that $\pi(\underline{x})=x$. Suppose that $\by\in\Sigma$ satisfies $\pi(\underline{y})\in Q((x,R))$.
We have that by Lemma \ref{agreement} for $1\leq i\leq r$, $\bx_i=\by_i$. Thus using that $r\geq n_m$ we have that
\begin{eqnarray*}
\overline{\mu}(Q(x,R))&=&\mu(\{\underline{y}\mid|\pi(\underline{y})\in Q(x,R)\})\\
&\leq&\mu(\{\by\mid\by_i=\bx_i\text{ for }1\leq i\leq r\})\\
&\leq&2^{-n_m+L_{m-1}}. 
\end{eqnarray*}  

\end{proof} 
In the final range we need to make use of the exponential separation condition that $\lambda$ satisfies. Note that for some of the range it would be possible to get better estimates but any finer analysis is unnecessary for case (1) of Theorem \ref{thm:main}. 
\begin{lemma}\label{Pablobound} 
For any $\epsilon>0$ there exists $C_1>0$ such that for $x\in \mathcal{C}$, $r_0(n_m)\leq r\leq r_{-1}(n_{m+1})$ and $R\in[C2^{-(r+1)},C2^{-r})$ we have the estimate
$$\overline{\mu}(Q(x,R))\leq C_12^{-n_m+L_{m-1}}\lambda^{(k(r)-(n_m+\ell_2(n_m))(1-\epsilon)}.$$

\end{lemma} 
 \begin{proof}
 Let $\epsilon>0$ and $\underline{x}\in\Sigma$ satisfy that $\pi(\underline{x})=x$.  Suppose that $\underline{y}\in\Sigma$ satisfies that  $\pi(\underline{y})\in Q(x,R).$
We have that by Lemma \ref{agreement} for all $1\leq i\leq r$ $y_i=x_i$. Moreover since $x\in \mathcal{C}$ we have that $x_i=z_{j-n_m}$ for all $r_0<j\leq k(r_0(n_m))=n_m+\ell_2(n_m)$. Thus if for some $r_0<j\leq k(r_0(n_m))=n_m+\ell_2(n_m)$ we have that $y_j\neq x_j$, and then
$$\mu([y_1,\ldots,y_r])=0.$$
Hence, we may suppose that $x_i=y_i$ for all $1\leq i\leq n_m+l_2(n_m)$. Thus we wish to find
$$V=\#\{(y_{n_m+l_2(n_m)+1},\ldots,y_{k(r)})\mid\pi([x_1,\ldots,x_{n_m+\ell_2(n_m)},y_{n_m+\ell_2(n_m)+1},\ldots,y_{k(r)}])\cap Q(x,R)\neq\emptyset\}.$$
Since $\lambda\in \mathcal{E}$ we can use Corollary \ref{cor:branching} to bound
$$V\leq C_1\lambda^{(1-\epsilon)(k(r)-(n_m+\ell_2(n_m)))}2^{k(r)-(n_m+\ell_2(n_m))}.$$
Putting this together we get that  
\begin{eqnarray*}
\overline{\mu}(Q(x,R))&\leq& V2^{-k(r)+L_m}\\
&\leq& V2^{-k(r)+n_{m}+\ell_2(n_m)}2^{-n_m+L_{m-1}}\\
&\leq& C_1\lambda^{(1-\epsilon)(k(r)-(n_m+\ell_2(n_m)))}2^{-n_m+L_{m-1}}
\end{eqnarray*}
which is what was wanted.
\end{proof}
We now need to consider what each of the above estimates gives as $m\to\infty$. Firstly, for $r_{-1}(n_m)\leq r\leq -n_m$ and $R_m\in(C2^{-(r+1)},C2^{-r})$ we obtain 
by Lemma \ref{boundby1} that, for any $x\in \mathcal{C}$, 
$$\frac{\log(\overline{\mu}(Q(x,R_m))}{\log R_m}\geq \frac{\log(2^{-r-L_{m-1}})}{\log C2^{-r-1}}=\frac{ r+L_{m-1}}{r+1}.$$
That is, there is a function $\delta(m)$, approaching $0$ with $m\to\infty$, for which
$$r+L_{m-1}\geq r-n_{m-1}\geq r-n_m(\delta(m))\geq r-r(\delta(m)\frac{\log 2}{-\log\lambda}).$$
Therefore in this region we have
$$\frac{\log(\overline{\mu}(Q(x,R_m))}{\log R_m}\geq 1+\mathit{o}(1).$$
In the second part we have $n_m\leq r\leq r_0(n_m)$ and $R_m\in [C2^{-(r+1)},C2^{-r}]$, so that by Lemma \ref{r_0bound} we have for $x\in \mathcal{C}$
$$\frac{\log(\overline{\mu}(Q(x,R_m))}{\log R_m}\geq \frac{\log 2^{-n_m+L_{m-1}}}{\log R_m}\geq \frac{\log 2^{-n_m+L_{m-1}}}{\log (C2^{-r_0(n_m)})}.$$
This gives 
$$\frac{\log(\overline{\mu}(Q(x,R_m))}{\log R_m}\geq \frac{\log 2}{-\log(\lambda\gamma)}+\mathit{o}(1).$$
Finally for the third part we have that for a fixed $\epsilon>0$, $r_0(n_m)\leq r\leq r_{-1}(n_{m+1})$ and $R_m\in [C2^{-(r+1)},C2^{-r}]$ by Lemma \ref{Pablobound}
$$\frac{\log(\overline{\mu}(Q(x,R_m))}{\log R_m}\geq \frac{\log(2^{-n_m+L_{m-1}})+\log(C_1\lambda^{(1-\epsilon)(k(r)-(n_m+\ell_2(n_m)))})}{\log 2^{-r_0(n_m)}+\log(\lambda^{k(r)-(n_m+\ell_2(n_m))})}$$
and thus in this region
$$\liminf_{m\to\infty} \frac{\log(\overline{\mu}(Q(x,R_m))}{\log R_m}\geq \min\{(1-\epsilon),\frac{\log 2}{-\log(\lambda\gamma)}\}+\mathit{o}(1).$$
Combining these estimates, letting $m\to\infty$, $\epsilon\to 0$ and using that $\frac{-\log 2}{\log(\lambda\gamma)}\leq 1$ gives, by part 1 of Lemma \ref{massdistribution}, that $\dim \mathcal{C}\geq \frac{-\log 2}{\log(\lambda\gamma)}$. Thus combining with Section 4.1 we get that $\dim R^*(z,\lambda)=\frac{-\log 2}{-\log\lambda\gamma}$ for all $\lambda\in (\frac{1}{2},\frac{1}{2\gamma})\cap\mathcal{E}$.This finished the proof of part (1) of Theorem \ref{thm:main}. 

\subsubsection*{Lower bound in Theorem \ref{thm:dynamical} (1)}
Finally for Theorem \ref{thm:dynamical} note that if we take $z\in F$ with $\pi(\underline{z})=z$ then the cantor set $\mathcal{C}$ constructed satisfies that $\mathcal{C}\subseteq \pi(R(\underline{z},(\ell_n)))$ where $\ell_n=\lceil n\frac{\log\lambda}{\log\gamma}\rceil$. Thus the above argument gives that
$$\dimH\pi(R(\underline{j},(\ell_n)))\geq \frac{-\log 2}{\log\lambda\gamma}$$
for any $\lambda\in [\frac{1}{2},\frac{1}{2\gamma}]\cap\mathcal{E}$ which combined with the upper bound in section 4.1 complete the proof of Part 1 of Theorem \ref{thm:dynamical}.

  \section{Proof of Lower bound for Theorem \ref{thm:main} (2): Unique expansion}\label{sec:uniquelower}

In this section we prove the lower bound for part (2) of Theorem \ref{thm:main}. In this case, we assume that the centre of the target $z\in F$ satisfies the unique expansion condition, Definition \ref{def:unique}. The argument holds for $\lambda\in ((2\gamma)^{-1},1)$ satisfying the exponential separation condition of Definition \ref{def:expsep}. The general outline of the argument is similar to the previous section in that we define a Cantor set of sparse returns, set up a measure $\bar \mu$ on it, and apply part 1 of Lemma \ref{massdistribution}. However, since in this region of $\lambda$ the dimension value is larger, there is much more impactful interaction between the target and the local behaviour of the measure $\bar\mu$, resulting in more different cases to investigate in the lemmas underneath.

Recall the dimension value from the statement and denote it by $t(\gamma)$, namely let
\[
t(\gamma):=2+\frac{\log \lambda}{\log 2} - \frac{\log \gamma}{\log(\gamma\lambda)}. 
\]
Let $(\ell_n)\in \N^\N$ and recall the notation 
$$R(\bz, (\ell_n)) :=\limsup_{n\to\infty}\{\bi\in \Sigma\mid \sigma^n(\bi)\in [\bz|_{\ell_n}]\}.$$

Let $\bz\in \Sigma, \lambda\in (1/2,1]$ and $\gamma\in (0,1)$. For each $n\in\N$ we let $\ell_n$ be the smallest integer such that $\lambda^{\ell_n}<\gamma^n$. Notice that in the notation of Definition \ref{def:alltheells} this corresponds to $\ell_n=\ell_2(n) $. (The notation $\ell_1(n)$ plays no role in this section, so we will suppress the subindex.)

We have that
\begin{equation}\label{inclusion}
\pi(R(\bz, (\ell_n)))\subset R^{*}(z,\gamma).
\end{equation}
Hence for lower bounds to $R^*(z, \gamma)$, we may just as well study this symbolic recurrence set. That is, if we can show that for some $\bz\in\Sigma$, $\gamma\in (0,1)$, $(\ell_n)$ as above and any $\lambda\in (1/2\gamma,1)\cap\mathcal{E}$ we have that
\begin{equation}\label{lbunique}
\dimH \pi( R(\bz, (\ell_n)) ) \geq 2+\frac{\log \lambda}{\log 2} - \frac{\log \gamma}{\log(\gamma\lambda)},  
\end{equation}
then the dimension result follows also for $R^*(\bz, \gamma)$.  The advantage of studying the dynamical version of the system is that this set is easier to analyse, and further, this method also gives the lower bound to Theorem \ref{thm:dynamical} immediately. 

\begin{remark}
The computations in this section work for any $\bz\in \Sigma$. However, they only match the upper bound from Section \ref{sec:upperbounds}, when $\overline{\pi}(\bz)$ has a unique $\lambda$-expansion. The case of $\pi(\bz)$ with multiple expansions is handled in Section \ref{sec:case3}. 
\end{remark}

We fix $\bz\in\Sigma$ and $\lambda\in (\frac{1}{2\gamma},1]\cap\mathcal{E}$.
To show \eqref{lbunique} we construct a Cantor subset and an appropriate mass distribution. Start by recalling
\[
R_{n}=\{\bi\in \Sigma\mid \sigma^n(\bi) \in [\bz|_{\ell_n}]\}.
\]

Let $(n_m)_{m\in \N}$ be any rapidly increasing sequence of integers (see underneath), and denote by 
\[
\mathcal{C}=\pi(\bigcap_{k=1}^\infty R_{n_m})\subset R^*, 
\]
the Cantor set of those points that return to the target at all of the times $(n_m)$. Denote $L_m=\sum_{i=1}^{m}\ell_{n_i}$, and assume that $n_m$ grows so fast that 
\begin{itemize}
\item $n_m>n_{m-1}+\ell_{n_{m-1}}$, and
\item $\frac{L_{m-1}}{n_m}\to 0$. 
\end{itemize}

Define a mass distribution $\mu$ on $\pi^{-1}(\mathcal{C})$ by setting for a finite word $n_{m-1}+\ell(n_{m-1})\le|\underline{x}|<n_m$ that if $x_{{n_m}+i}\neq z_i$ for some $1\leq i\leq \ell_{n_m}$ and $n_m+i\leq |x|$ then $\mu([\underline{x}])=0$. Otherwise if $n_{m-1}+\ell(n_{m-1})\leq|\underline{x}|<n_m$ we set
\[
\mu[\underline{x}]=(\tfrac 12)^{|\bx|-L_{m-1}}, 
\]
and for a finite word $n_m\le |\underline{x}|<n_m+\ell_{n_m}$, 
\[
\mu[\bx]=(\tfrac 12)^{n_m - L_{m-1}}. 
\]
This extends to a Borel probability measure $\mu$ on $\pi^{-1}(\mathcal C)$ by Carath\'eodory's Extension Theorem. Denote $\overline{\mu}=\pi_*\mu$ and note that $\overline{\mu}(\mathcal{C})=1$.

Let $R>0$ and $x\in \pi^{-1}(\mathcal{C})$. We would like to give upper bounds on $\overline{\mu}(Q(x, R))\lesssim R^s $ which will enable us to bound below $\liminf\frac{\log\overline{\mu}(Q(x,R))}{\log R}$ which will then give a lower bound on the dimension of $\mathcal{C}$ by Lemma \ref{massdistribution}.

Let $R\in (C2^{-(r+1)}, C2^{-r})$ for some $r\in \N$. Let $m$ be the closest return to $r$, that is, assume that $n_{m-1}+\ell_{n_{m-1}}\le r < n_m+\ell_{n_m}$. The arguments presented throughout will depend crucially on the size of $(\tfrac 12)^r$ relative to $\lambda^{n_m}, \lambda^{n_m}\gamma^{n_m}$. For $r, n\in \N$ we recall and define the following notations: 
\begin{itemize}
\item $\ell_n$: the smallest integer $l$ such that $\lambda^{l}< \gamma^n$, 
\item $k(r)\in \N$: the smallest integer $k$ for which $\lambda^{k}<(\tfrac 12)^r$,
\item $r_0\in \N$: the smallest integer for which $k(r_0)\geq\ell_{n_m}+n_m$, 
\item $r_{-1}\in \N$:  the largest integer  for which $(\tfrac 12)^{r_{-1}}\geq \lambda^{n_m}$,
\item $r_{1}\in\N$: the largest integer for which $(\tfrac 12)^{r_{1}}\geq \lambda^{n_{m+1}}$,
\item $C$: the constant from Lemma \ref{agreement}, 
\item $\mathit{o}(1)$ is a function that approaches $0$ with $R\to 0$ (or $r\to \infty$, or $m\to \infty$ as all of these happen simultaneously). 
\end{itemize}
In the lemmas underneath, we go through the possible relative positions of $r$ to these other quantities. Note that since $\lambda>\frac{1}{2\gamma}$ we have that $r_{-1}<r_0<n_m<n_m+\ell(n_m)$. We will now fix $m\in\N$ and cover all the possible cases.

\begin{lemma}\label{lem:dynamrlessr_0}
There is some $C_1>0$ such that the following holds: Let $\epsilon>0$, $x\in \mathcal{C}$, $r_{-1}\le r<r_0$ and $R\in [C2^{-(r+1)},C2^{-r})$. Then 
\[
\overline{\mu}(Q(x, R))\le C_1 (\tfrac 12)^{{r-L_{m-1}}}\lambda^{(n_m-r)(1-\epsilon)}.  
\]
\end{lemma}
\begin{proof}
Notice that in order for a sequence $\bi\in \Sigma$ to satisfy $\pi(\bi)\in Q=Q(x, R)$ we have $\bx|_r=\bi|_r$ by Lemma \ref{agreement}. We will have to give an estimate for the number of next $n_m-r$ symbols in $\bi$ which would still let $\pi(\bi)\in Q$, and after that until $n_m+\ell_{n_m}$ they will have to agree with $\bz$. By the definition of $\mu$, their mass does not change after this level. Denote 
\[
N=\#\{\bi\in \Sigma^{n_m}\mid \pi[\bi]\cap Q(x, R)\neq\emptyset\}.
\]
By the above discussion we have the estimate
\[
\overline{\mu}(Q)\le N(\tfrac 12)^{n_m-L_{m-1}}. 
\]

The rest of the argument focuses on bounds for $N$. Note that in order for $\pi[\bi]\cap Q\neq \emptyset$, we need 
\[
\sum_{j=1}^{n_m}(\bi_j - \bx_j)\lambda^j\le C(\tfrac 12)^r\le (\tfrac 12)^r,  
\]
but since we know that $\bi_j=\bx_j$ for all $j=1, \dots, r$, this translates to 
\[
\sum_{j=r}^{n_m}(\bi_j - \bx_j)\lambda^j\le (\tfrac 12)^r,   
\]
and further to 
\[
\sum_{j=r}^{n_m}(\bi_j - \bx_j)\lambda^{j-r}\le (\tfrac 12)^r\lambda^{-r}.   
\]
Note that since $r\ge r_1$, we have $(\tfrac 12 )^r\le \lambda^{n_m}$, and hence for some $x'\in [0,1]$, 
\[
N\le \#\{\bi \in \Sigma_{n_m-r}\mid \bar\pi[\bi]\cap B(x', \lambda^{n_m-r})\}
\]
To estimate this number, we can apply Corollary \ref{cor:branching}, and obtain
\begin{equation}\label{eq:N}
N\le  C_1\frac{\lambda^{(n_m-r)(1-\epsilon)}}{(\tfrac 12)^{n_m-r}}. 
\end{equation}
Together with the argument from the beginning of the proof this implies 
\[
\overline{\mu}(Q)\le C_1(\tfrac12)^{n_m-L_{m-1}}\frac{\lambda^{(n_m-r)(1-\epsilon)}}{(\tfrac 12)^{n_m-r}}. 
\]
\end{proof}

\begin{remark}
 We can apply Lemma \ref{lem:dynamrlessr_0}, to obtain that 
\begin{eqnarray*}
\frac{\log \overline{\mu}(Q(x,R))}{\log R}&\geq&\frac{(n_m-r)(1-\epsilon)\log\lambda + (r-L_{m-1})\log (\tfrac 12)}{(r+1)\log (C(\tfrac 12))}+\frac{\log C_1}{(r+1)\log (C(\tfrac 12))}\\
&\geq&\frac{(n_m-r)(1-2\epsilon)\log \lambda}{(r+1)\log C(\tfrac 12)}+\frac{(r-L_{m-1})\log (\tfrac 12)}{(r+1)\log (C(\tfrac 12))}+o(1)\\
&\geq& \frac{n_m(1-\epsilon)\log\lambda}{(r+1)\log (C(\tfrac 12))}+\frac{r(1-\epsilon)\log\lambda}{(r+1)\log (2C)}+\frac{(r-L_{m-1})\log(\tfrac 12)}{(r+1)\log (C(\tfrac 12))}+o(1).
\end{eqnarray*}
We have that $L_{m-1}=\mathit{o}(r_{-1})$ and $n_m/r\geq \frac{n_m}{r_0}\to\frac{\log 2}{-\log(\lambda\gamma)}$ as $m\to\infty$. Thus we have 
\begin{eqnarray*}
\frac{\log \overline{\mu}(Q(x, R))}{\log R}&\geq& \frac{\log 2}{-\log(\lambda\gamma)}\frac{(1-\epsilon)\log\lambda}{-\log 2}+\frac{(1-\epsilon)\log\lambda}{\log 2}+1+\mathit{o}(1)\\
&\geq& \frac{(1-\epsilon)\log\lambda}{\log(\lambda\gamma)}+\frac{(1-\epsilon)\log\lambda}{\log 2}+1+\mathit{o}(1)\\
&=& (1-\epsilon)+\frac{(1-\epsilon)\log\gamma}{\log(\lambda\gamma)}+\frac{(1-\epsilon\log\lambda}{\log 2}+1+\mathit{o}(1)\\
&=&2-\epsilon+\frac{(1-\epsilon)\log\gamma}{\log(\lambda\gamma)}+\frac{(1-\epsilon\log\lambda}{\log 2}+\mathit{o}(1).
\end{eqnarray*}
Taking $\epsilon$ to $0$ gives the dimension value $t(\gamma)$ from Theorem \ref{thm:main} (2). 
\end{remark}

\begin{lemma}\label{lma:r0rn}
There exists $C_2>0$ such that for $\epsilon>0$, $x\in\mathcal{C}$, $r_0\leq r<n_m$ and $R\in [C2^{-r-1},C2^{-r})$ we have that 
\[
\overline{\mu}(Q(x, R))\le C_2\lambda^{(-r-\ell_{n_m})}(\tfrac12)^{r+r(1-\epsilon)-L_{m-1}}. 
\]
\end{lemma}

\begin{proof}
We let $\underline{x}\in\Sigma$ satisfy that $\pi(\underline{x})=x$ and note that as before for $\bi\in\Sigma$ to satisfy $\pi(\bi)\in Q(x, R)=Q$ we need to have $\bi|_r=\bx|_r$ by definition of $C$. We then need an estimate for the number of $\bi|_r=\bx|_r$ with also $\bi|_{n_m}$ such that $\pi[\bi|_{n_m}]\cap Q\neq \emptyset$. This quantity we have just defined as $N$ and estimated from above in Lemma  \ref{lem:dynamrlessr_0}, see \eqref{eq:N}.

From there, between the times $n_m$ and $n_m+\ell_{n_m}$, the mass does not change. But, since $r>r_0$, not all of the mass of the cylinders on level $n_m+\ell_{n_m}$ is inside $Q$. That is, for a projected cylinder $\pi[\bi|_{n_m+\ell_{n_m}}]$ that intersects $Q$, we need to find the proportion of mass within $Q$. Fix some such $\bi|_{n_m+\ell_{n_m}}$. Denoting $x_1$ the first coordinate of $x$, we have:
\[
\pi^{-1}(Q\cap \pi[\bi|_{n_m+\ell_{n_m}}])\le\bar\pi^{-1} (B(x_1, C(\tfrac 12)^{r})\cap \bar\pi[\bi|_{n_m+\ell_{n_m}}]). 
\]
(Recall that $\bar\pi$ is the coding map from $\bar\pi: \Sigma \to [0,1]$.) Hence we can apply Corollary \ref{cor:branching} to obtain the bound: 
\[
\overline{\mu}(Q(x,R)\cap \pi[\bi|_{n_m+\ell_{n_m}}])\le C_1\left(\frac{(\tfrac 12)^r}{\lambda^{\ell_{n_m}+n_m}}\right)^{1-\epsilon}.
\]
Combining all of the above, we are left with 
\begin{eqnarray*}
\overline{\mu}(Q(x,R))&\leq& (\tfrac 12)^{n_m-L_{m-1}}NC_1\left(\frac{(\tfrac 12)^r}{\lambda^{\ell_{n_m}+n_m}}\right)^{1-\epsilon}\\
&\leq& C_1^2\lambda^{(n_m-r-\ell_{n_m}-n_m)(1-\epsilon)}(\tfrac12)^{r+r(1-\epsilon)-L_{m-1}}\\
&\leq& C_1^2\lambda^{(-r-n_m)(1-\epsilon)}(\tfrac12)^{r+r(1-\epsilon)-L_{m-1}}\\
&\leq& C_2\lambda^{-r-n_m}(\tfrac12)^{r+r(1-\epsilon)-L_{m-1}} 
\end{eqnarray*}
for $C_2=C_1^2$, which was the claim. 
\end{proof}

\begin{remark}
As before we look at what this means for the local dimension as $m$ approaches infinity. We have  
\begin{align*}
\frac{\log \overline{\mu}(Q(x,R))}{\log R}&\geq
\frac {\log C_2+(-r-\ell_{n_m})(\log\lambda)-(r(2-\epsilon)-L_{m-1})\log (2)}{\log C-(r+1)\log 2}\\
&\geq  \frac{\log (C_2)}{\log C-(r+1)\log 2}+\frac{(-r-\ell_{n_m})(\log\lambda)}{\log C-(r+1)\log 2}-\frac{(r(2-\epsilon)-L_{m-1})\log (2)}{\log C-(r+1)\log 2}. 
\end{align*}
The first part of the sum goes to $0$ as $m\to\infty$. For the second part we use that $\frac{\log\lambda}{\log 2}<0$ along with 
$$\frac{r+\ell_{n_m}}{r}\leq\frac{r_0+\ell_{n_m}}{r_0}\to \frac{\log(\lambda\gamma)-\frac{\log(\gamma)\log(2)}{\log\lambda}}{\log(\lambda\gamma)}\text{ as }m\to\infty$$
and finally for the third part $\lim_{m\to\infty} \frac{L_{m-1}}{r}=0$. This gives the following bound 
\begin{eqnarray*}
\frac{\log \overline{\mu}(Q(x,R))}{\log R} &\geq& 2-\epsilon+\left(\frac{\log\lambda}{\log 2}\right)\left(\frac{\log(\lambda\gamma)-\frac{\log(\gamma)\log(2)}{\log\lambda}}{\log(\lambda\gamma)}\right)+\mathit{o}(1)\\
&\geq&2-\epsilon-\frac{\log\gamma}{\log(\lambda\gamma)}+\frac{\log\lambda}{\log 2}+\mathit{o}(1)\\
&\geq& t(\gamma)-\epsilon+\mathit{o}(1).
\end{eqnarray*}
Again taking $m\to\infty$ and then $\epsilon\to 0$  yields the desired bound.
\end{remark}

\begin{lemma}\label{lma:r0rnl}
There exists $C_3>0$ satisfying the following: Let $\epsilon>0$,  $x\in \mathcal{C}$, $r_0\le n_m< r\le n_m+\ell(n_m)$, and $R\in [C2^{-(r+1)},C2^{-r})$
\[
\overline{\mu}(Q(x,R))\le C_3\lambda^{(-\ell_{n_m}-n_m)}(\tfrac12)^{n_m+r - r\epsilon-L_{m-1}}. 
\]
\end{lemma}

\begin{proof}
Since $r>n_m$, the cube $Q$ only intersects one of the projected cylinder sets $\pi[\bi|_{n_m+\ell_{n_m}}]$ from level $\ell(n_m)+n_m$, but since $r>r_0$, it does not contain it. We need to estimate the proportion of measure $\bar\mu(\pi[\bi|_{n_m+\ell_{n_m}}])$ that is inside $Q$. Here, as before in the proof of Lemma \ref{lma:r0rn}, we make use of the fact that 
\[
\pi^{-1}(Q\cap \pi[\bi|_{n_m+\ell_{n_m}}])\le \bar\pi^{-1} (B(x_1, C(\tfrac 12)^{r})\cap \bar\pi[\bi|_{n_m+\ell_{n_m}}]). 
\]
Indeed, as above in the proof of Lemma \ref{lma:r0rn}, an application of Corollary \ref{cor:branching} gives 
\begin{eqnarray*}
\overline{\mu}(Q(x,R))&\le&(\tfrac 12)^{n_m-L_{m-1}}C_1\frac{(\tfrac 12)^{r(1-\epsilon)}}{\lambda^{(\ell_{n_m}+n_m)(1-\epsilon)}}\\
&\le& C_1\lambda^{(-\ell_{n_m}-n_m)(1-\epsilon)}(\tfrac12)^{n_m+r - r\epsilon-L_{m-1}}\\
&\leq& C_3\lambda^{(-\ell_{n_m}-n_m)}(\tfrac12)^{n_m+r - r\epsilon-L_{m-1}}.
\end{eqnarray*}

\end{proof}

\begin{remark}
 The estimate from Lemma \ref{lma:r0rnl} gives that
$$
\frac{\log \overline{\mu}(Q(x,R))}{\log R}\ge \frac{(-\ell_{n_m}-n_m)\log \lambda+(n_m+r-r\epsilon-L_{m-1})\log(\tfrac 12)}{-C(r+1)\log 2}+\frac{\log C_3}{\log R}.$$
We can now use that $r\le \ell(n_m)+n_m$ to bound 
$$\frac{(-\ell_{n_m}-n_m)\log \lambda}{-C(r+1)\log 2}=\frac{(\ell_{n_m}+n_m)\log \lambda}{C(r+1)\log 2}\geq\frac{\log\lambda}{\log 2}$$
and
$$\frac{n_m+r(1-\epsilon)}{r}=\frac{n_m+r}{r}-\epsilon\geq \frac{2n_m+\ell(n_m)}{n_m+\ell_{n_m}}-\epsilon.$$
This gives 
\begin{eqnarray*}
\frac{\log \overline{\mu}(Q(x,R))}{\log R}&\geq& 1+\frac{\log\lambda}{\log 2}+\frac{1}{1+\frac{\log\gamma}{\log\lambda}}-\epsilon+\mathit{o}(1)\\
&\geq&2+\frac{\log\lambda}{\log 2}-\frac{\frac{\log\gamma}{\log\lambda}}{1+\frac{\log\gamma}{\log\lambda}}-\epsilon+\mathit{o}(1)\\
&\geq&2+\frac{\log\lambda}{\log 2}-\frac{\log\gamma}{\log(\gamma\lambda)}-\epsilon+\mathit{o}(1)\\
&\geq& t(\gamma)-\epsilon+o(1)
\end{eqnarray*} 
and letting $m\to\infty$ and then $\epsilon\to 0$ completes the bound in this region.
\end{remark}

\begin{lemma}\label{lem:lower31}
There exists $C_4>0$ satisfying the following: Let $\epsilon>0$, let $x\in\mathcal{C}$, $n_{m}+\ell_{n_m}\le r<r_1$, and $R\in [C2^{-{r+1}},C2^{-r})$ 
\[
\overline{\mu}(Q(x,R))\le C_4(\tfrac 12)^{r(2-\epsilon)-\ell_{n_{m}}-L_{m-1}}\lambda^{-r}. 
\]
\end{lemma}
\begin{proof}
Firstly, any $\bi\in \Sigma$ satisfying $\pi(\bi)\in Q(x, R)$ will have to have $\bi|_r=\bx|_r$ as above, by Lemma \ref{agreement}. Hence, we first set out to estimate the value of $\mu[\bx|_r]$. Here, 
\begin{equation}\label{eq:eq1}
\mu[\bx|_r]=(\tfrac 12)^{r-L_m}=(\tfrac 12)^{r-L_{m-1}-\ell_{n_m}}. 
\end{equation}
Notice that ultimately, any error from $L_{k-1}$ is going to be `small' relative to $r$, but depending on the position of $r$, $\ell_{n_m}$ could be `large'. That is the reason we have divided the power up in this way.

Only some out of the cylinders within $\pi[\bx|_{r}]$ will contribute to $Q(x, R)$, and we now estimate this proportion. Note that in order for some $\bi\in \Sigma$ to contribute to $\overline{\mu}(Q(x,R))$, the $x$-coordinate of $\pi(\bi)$ needs to be close enough to the $x$-coordinate of $\pi(\bx))$. In particular $[\bx|_{r}]$ has to satisfy
\[
\sum_{j=1}^\infty (\bi_j - \bx_j)\lambda^j\le C(\tfrac 12)^r\le (\tfrac 12)^r, 
\]
but from earlier we also know that $\bi|_r=\bx|_r$ so this becomes 
\[
\sum_{j=r}^\infty (\bi_j - \bx_j)\lambda^j\le (\tfrac 12)^r \Leftrightarrow \sum_{j=r}^\infty (\bi_j - \bx_j)\lambda^{j-r}\le (\tfrac 12)^r\lambda^{-r}. 
\]
Hence those words $\bi$ contributing to the measure $\overline{\mu}(Q(x,R))$ are those of the form $\bx|_r\bi$ with 
\[
\bi\in B:=\{\bi \in \Sigma \mid \sum_{j=1}^\infty (\bi_j - \sigma^r(\bx)_j)\lambda^{j}\le R\lambda^{-r}\}. 
\]
Since $C<1$, we have that $R\lambda^{-r}\leq \lambda^{k(r)-r}$. Hence, to estimate the measure Contributing we can use corollary \ref{cor:branching} to get that there is a constant $C_3>0$ with
$$\#\{(j_1,\ldots,j_{k(r)-r})\mid[j_1,\ldots,j_{k(r)-r}]\cap B\neq\emptyset\}\leq C_42^{k(r)-r}\lambda^{(k(r)-r)(1-\epsilon)}.$$
Combining this with \eqref{eq:eq1} gives a bound 
\begin{eqnarray*}
\overline{\mu}(Q((x,R))&\le& C_4(\tfrac12)^{r-\ell_{n_m}-L_{m-1}}\lambda^{(k(r)-r)(1-\epsilon)}2^{(k(r)-r)}2^{-k(r)+r}\\
&\leq&C_4(\tfrac 12)^{r-\ell_{n_m}-L_{m-1}}\lambda^{(k(r)-r)(1-\epsilon)}\\
&\leq&C_4(\tfrac 12)^{r-\ell_{n_m}-L_{m-1}}2^{-r(1-\epsilon)}\lambda^{-r}\\
&\leq&C_4(\tfrac 12)^{(2-\epsilon)r-\ell_{n_m}-L_{m-1}}\lambda^{-r}. 
\end{eqnarray*}

\end{proof}

\begin{remark}
Notice that, for the upper bound from Lemma \ref{lem:lower31} we get
\begin{eqnarray*}
\frac{\log \overline{\mu}(Q(x, R))}{\log R}&\ge& \frac{\log (\tfrac 12)^{r(2-\epsilon)-\ell_{n_{m}}-L_{{m-1}}}\lambda^{-r}}{C\log (\tfrac 12)^{r+1}}+\frac{\log C_4}{\log R}\\
&\geq& \frac{-r(2-\epsilon)(\log 2)}{C\log (\tfrac 12)^{r+1}}+\frac{\ell_{n_m}\log 2}{C\log (\tfrac 12)^{r+1}}+\frac{L_{m-1}\log 2}{C\log (\tfrac 12)^{r+1}}-\frac{r\log\lambda}{C\log (\tfrac 12)^{r+1}}+o(1).
\end{eqnarray*}
We can use that 
$$\frac{\ell_{n_{m}}}{r}\leq \frac{\ell_{n_m}}{n_m+\ell_{n_m}}\to\frac{\log \gamma}{\log(\gamma\lambda)}$$
to get that

$$\frac{\log \overline{\mu}(Q(x, R))}{\log R}\geq 2-\epsilon-\frac{\log \gamma}{\log(\gamma\lambda)}+\frac{\log\lambda}{\log(2)}+\mathit{o}(1)$$
and letting $\epsilon\to 0$  yields
$$\liminf_{m\to\infty}\frac{\log \overline{\mu}(Q(x, R))}{\log R}\geq t(\gamma).$$
\end{remark}

The lemmas and remarks throughout the section cover all possible values of $R$. Hence, we have now checked that for all $x\in \mathcal C$ 
$$\liminf_{R\to 0}\frac{\log \mu(Q(x, R)}{\log R}\geq t(\gamma)$$
and thus by part 1 of Lemma \ref{massdistribution} we have
\begin{equation*}
\dim R^*(z,\lambda) \ge t(\gamma). 
\end{equation*}
We can combine this with the upper bound from Section \ref{sec:upperbounds} for $\lambda\in (1/2\gamma,1)\cap\mathcal{E}$, and for those $z\in  F$ for which $\proj_1(z)$ has a unique expansion, to achieve part (2) of Theorem \ref{thm:main}.
\begin{remark}\label{rem:sharp}
As mentioned above the lower bound for this section holds for any choice of centre $z$. Hence, letting $z\in F$, $\lambda\in (1/2,1)\cap\mathcal{E}$, and letting $(\gamma_n)$ satisfy that $\gamma_n\in (0,1)$ and $\lim_{n\to\infty}\gamma_n^{\frac{1}{n}}=1$, we obtain for the set
$$R^*=\{x\in F\mid E^n(x)\in Q(z,\gamma_n)\text{ for infinitely many }n\}$$
that 
$$\dim R^*\geq\lim_{\gamma\to 1}\left(2+\frac{\log 2}{\log\lambda}-\frac{\log\gamma}{\log(\gamma\lambda)}\right)=2+\frac{\log\lambda}{\log 2}=\dim F.$$
Hence, 
$$\dim R^*=\dim F.$$ 
This demonstrates that sub-exponential speed of shrinking leads to sets $R^*$ with full dimensions. It also can be seen from part 1 of Theorem \ref{thm:main} that superexponential speed of shrinking targets leads to sets $R^*$ with zero dimension. 
\end{remark}

\section{Lower bound in Theorem \ref{thm:main} (3): Typical centre}\label{sec:case3}

In this section, we consider the third case of Theorem \ref{thm:main}. We will consider multiple parameters $\lambda$ simultaneously.  The dependence on $\lambda$ as a parameter will, for example, be relevant to the definition of $\ell_2$ and hence we will write $\ell_2(n, \lambda)$ to emphasise the dependence. We consider $\lambda>(2\gamma)^{-1}$. Just as in Section \ref{sec:uniquelower}, we need to define a Cantor subset $\mathcal{C}$ of $R^*$ and a corresponding mass distribution $\mu$ in order to estimate the dimension from below. The difference between parts (2) and (3) in Theorem \ref{thm:main} is that in part (2) the centre of the target $z$ had a unique $\lambda$ expansion, but in part (3) we are dealing with a centre $z$, generic for $\onf$ on $F$, and which will in general have multiple expansions. Here we need to use both exponential separation and transversality techniques, and consequently, our results hold for Lebesgue almost every $\lambda$ within the region of transversality, that is, for $\lambda\in(\tfrac 1{2\gamma}, \bar \lambda)$ (see Definition \ref{def:region}). In this case the centre of the shrinking target can affect the dimension. 

The strategy is still to consider a Cantor set where the returns are at sparse, predetermined times $(n_m)$, together with an appropriate well-distributed measure on this set. However, we cannot use the exact same choice of $\mathcal{C}$ and the measure $\mu$ as we did before. This is a consequence of the following: For points $\bi\in R_n$, the choices of digits between times $n_m$ and $n_m+\ell_1(n_m)$ are fixed: $\bi|_{n_m}^{n_m+\ell_1(n_m)}$ has to agree with $\underline z|_{\ell_1(n_m)}$. However, as explained in Remark \ref{rem:structure}, between $n_m+\ell_1(n_m)$ and $n_m+\ell_2(n_m, \lambda)$ there are potentially multiple possibilities for the digits in $\bi$, and indeed, both the value of $\ell_2(n_m, \lambda)$ and the exact distribution of branching in this range depend on $\lambda$. These appear as parameters in the definition of the measure below. 

\subsection{Preliminaries}

Let us make some choices. Fix $\gamma>0$, and $\delta>0$. Let $\tfrac 1{2\gamma}<\lambda_0<\lambda_1<\bar \lambda$. Recall that $\nu$ is the $(\tfrac 12,\tfrac 12)$-Bernoulli measure on $\Sigma$.

Consider a sequence $(n_m)$ to be fixed -- the conditions this sequence will satisfy are described underneath in Section \ref{sec:defC}. For a fixed $m\in\N$, we will use the notation as in the previous section but with a little adjustment to account for the fact that we need to work uniformly with $\lambda\in [\lambda_0,\lambda_1]$. We set:
\begin{enumerate}
\item
$r_{-1}=r_{-1}(n_m)$: the smallest integer $l$ such that $2^{-l}\leq \lambda_0^{n_m}$;
\item
$r_{0}=r_0(n_m)$: the smallest integer $l$ such that $2^{-l}\leq \lambda_0^{n_m}$;
\item
$r_{1}=r_{1}(n_m)$: the smallest integer $l$ such that $2^{-l}\leq \lambda_0^{n_{m+1}}$;
\item
$\ell_2(n_m,\lambda)$: the smallest integer $l$ such that $\lambda^{l}\leq \gamma^{n_m}$;
\item
$\ell_1(n_m)$: the smallest integer $l$ such that $2^{-l}\leq\gamma^{n_m}$;
\item
$\xi:=\frac{\log\gamma\log(\lambda_0 2)}{\log \lambda_0\log 2}$;
\item
$s:=\frac{2\log 2+\log\lambda_0}{\log(2/\gamma)}$;
\item
$\eta(\lambda_0,\lambda_1):=\left(\frac{\log \lambda_0}{\log \lambda_1}-1\right)+(s+1)\left(\frac{\log\gamma}{\log \lambda_1}-\frac{\log\gamma}{\log \lambda_0}\right)$.

\end{enumerate}
 It is helpful to keep in mind, and will be used repeatedly in the computations, that for large $m$, $\xi n_m$ is approximately $\ell_2(n_m, \lambda_0)-\ell_1(n_m)$.
 We begin with a lower bound for a typical centre of the number of cylinder sets which intersect the target, which is an immediate consequence of Lemma \ref{lem:msrballlower}.  
\begin{lemma}\label{lem:lowerboundbranching}
For all $\epsilon>0$ there exists $\Lambda=\Lambda_\epsilon\subset [\lambda_0,\lambda_1]$ such that
\begin{enumerate}
\item
$\mathcal{L}(\Lambda)>\lambda_1-\lambda_0-\epsilon/2$
\item
For $\nu$-almost all $\omega$ there exists $C>0$ and $N(\omega)\in\N$ such that for all $n\geq N$ and $\lambda\in\Lambda$,
$$\#\{(\tau_{n+1},\ldots\tau_{n+\lceil\xi n\rceil})\mid\left|\sum_{i=1}^{\lceil \xi n\rceil} (\tau_{n+i}\lambda^i-\omega_{n+i}\lambda^i)\right|\leq\lambda^{\xi n}\}\geq C(2\lambda^{(1+\epsilon)})^{\xi n}.$$ 
\end{enumerate}
\end{lemma}
\begin{proof}
This follows immediately by applying Lemma \ref{lem:msrballlower} with the choice of $\xi$.
\end{proof}

Over the course of the argument we will make use of the following lemmas which also are straightforward consequences of results from Section \ref{sec:prels}. 
\begin{lemma}\label{lem:unifpablo}
Let $0<\epsilon<\lambda_1-\lambda_0$. There exists $C\in\N$ such that 
$$\mathcal{L}\{\lambda\in [\lambda_0,\lambda_1]\mid\overline{\nu_{\lambda}}(I)\leq C|I|^{1-\epsilon}\text{  for all }I\subseteq [0,1]\}\geq \lambda_1-\lambda_0-\epsilon/2.$$
\end{lemma}
\begin{proof}
We know from Theorem \ref{thm:pablo} that 
$$\mathcal{L}(\cup_{C\in\N}\{\lambda\in [\lambda_0,\lambda_1]\mid\overline{\nu_{\lambda}}(I)\leq C|I|^{1-\epsilon}\text{  for all }I\subseteq [0,1]\})=\lambda_1-\lambda_0.$$
As this is a nested union the result follows. 
\end{proof}

\begin{lemma}\label{lem:ambiguous}
Fix $z\in \Lambda$ so that Lemma \ref{lem:lowerboundbranching} holds for $\bz$. Let $\epsilon>0$, and let $\lambda\in \Lambda_\epsilon\subset [\lambda_0, \lambda_1]$. Let $n=n_m$ for some $m$, and denote $\ell_1=\ell_1(n), \ell_2=\ell_2(n, \lambda)$. Now, 
\[
\# D_{\ell_2-\ell_1}(\sigma^{\ell_1}(\bz), \lambda^{\ell_2-\ell_1})\ge (2\lambda)^{\xi n}\lambda^{\epsilon\xi n}. 
\]  
\end{lemma}
 \begin{proof}
 By choice of $z\in \Lambda$, Lemma \ref{lem:lowerboundbranching} holds for $\bz\in \Sigma$ with $N(\bz)<\ell_1(n_1)$. Now, by definition of $\ell_2, \ell_1$, since $\lambda>\lambda_0$, $\xi n\ge \ell_2(\lambda, n)-\ell_1(n)$ and
\begin{align*}
\# D_{\ell_2-\ell_1}&(\sigma^{\ell_1}(\bz), \lambda^{\ell_2-\ell_1})\\
 &=\#\{|\bj|=\ell_2-\ell_1\mid |\sum_{i=1}^{\ell_2-\ell_1}(j_i - z_{i+\ell_1})\lambda^i|<\lambda^{\ell_2-\ell_1}\}\\
 &\ge \#\{|\bj|=\lceil\xi n\rceil\mid |\sum_{i=1}^{\lceil n\xi\rceil}(j_i - z_{i+\ell_1})\lambda^i|<\lambda^{\xi n}\}\\
 &\ge C(2\lambda)^{n\xi}\lambda^{\xi \epsilon n}.  
\end{align*} 
 \end{proof}
 
 \subsection{Definition of the Cantor set and the measure on it}\label{sec:defC}

\subsubsection*{Definition of the Cantor set $\mathcal{C}$} Fix $z\in \Lambda$ so that Lemma \ref{lem:lowerboundbranching} above holds for the $\bz\in \Sigma$ such that $\pi(\bz)=z$. Also fix $\lambda$ for the time being. Denote by $N=N(\bz)$ the number $N$ from Lemma \ref{lem:lowerboundbranching}. Let $Q_n=Q(z, \gamma^n)$. Recall that 
\[
R_n^*=\{x\in \Lambda\mid E^n(x)\in Q_n\}. 
\]
Let $n_1$ be so large that $\ell_1(n_1)>N(z)$. In what follows, we will determine a fast increasing subsequence $(n_m)$, satisfying, to begin with, $n_m+\ell_2(n_m, \lambda_1)<n_{m+1}$ (see underneath). Set 
\[
\mathcal{C}=\bigcap_{m=1}^\infty R_{n_m}^*\subset R^*. 
\]

\subsubsection*{Definition of a measure on $\mathcal{C}$}
For $m\in\N$ denote 
\[
D_m(\underline x, \rho)=\{\bi \in \Sigma_m\mid \pi[\bi]\cap B(\bx, \rho)\neq \emptyset\}. 
\]
As outlined in Remark \ref{rem:structure}, it will be of particular interest, when is $\pi[\bi|_{n_m}^{n_m+\ell_2(n_{m}, \lambda)}]\cap Q_n\neq \emptyset$.

For fixed $z$ and $\lambda$ as above, define $\mu_\lambda$ on $\pi^{-1}(\mathcal{C})$ by setting for $\bx\in \Sigma^*$ with  $n_{m-1}+\ell_2(n_{m-1}, \lambda)\le |\bx|\le n_{m}$, 
\[
\mu_\lambda[\bx]=\mu_\lambda[\bx|_{n_{m-1}+\ell_2(n_{m-1}, \lambda)}](\tfrac 12)^{|\bx|-n_{m-1}+\ell_2(n_{m-1}, \lambda)}. 
\]
That is, between target regions we distribute the mass evenly between the two digits. For $n_m\le |\bx|\le n_m+\ell_1(n_m)$, set 
\[
\mu_\lambda[\bx]=\mu_\lambda[\bx|_{n_m}], 
\]
that is, concentrate the measure on the unique choice of digits in this region. Then for $|\bx|= n_m+\ell_2(n_{m}, \lambda)$ such that 
\[
\bx|_{n_m}^{n_m+\ell_2(n_m, \lambda)} \in D_{\ell_2(n_{m}, \lambda)-\ell_1(n_m)}(\sigma^{\ell_1(n_m)}(\bz), \lambda^{\ell_2(n_{m}, \lambda)-\ell_1(n_m)}),
\] 
set (inductively)
\begin{equation}\label{eq:ambiguous}
\mu_\lambda[\bx]= \frac{\mu_\lambda[\bx|_{n_m}]}{\#D_{\ell_2(n_{m}, \lambda)-\ell_1(n_m)}(\sigma^{\ell_1(n_m)}(\bz), \lambda^{\ell_2(n_{m}, \lambda)-\ell_1(n_m)})}. 
\end{equation}
That is, we subdivide the mass evenly between any branches in this \emph{ambiguous} region. This definition also implicitly determines the weights of cylinders with lengths $n_m+\ell_1(n_m)\le |\bx|\le n_m+\ell_2(n_{m}, \lambda)$ and extends to a measure on $\Sigma$ by Carath\'eodory's extension theorem. The measure is supported on $\pi^{-1}(\mathcal{C})$ by compactness of $\mathcal{C}$. We can also estimate that for any $\lambda\in [\lambda_0,\lambda_1]$, by Lemma \ref{lem:ambiguous}
$$\#D_{\ell_2(n_{m}, \lambda)-\ell_1(n_m)}(\sigma^{\ell_1(n_m)}(\bz), \lambda^{\ell_2(n_{m}, \lambda)-\ell_1(n_m)})\geq (2\lambda)^{\xi n}\lambda^{\epsilon\xi n}$$
and so for $|\underline{x}|=n_m+\ell_2(n_{m}, \lambda)$ we have
\begin{equation}\label{eq:measestimate}
\mu_\lambda[\bx]\leq \mu_\lambda[\bq|_{n_m}]\lambda^{-\epsilon\xi n}(2\lambda)^{-\xi n}.
\end{equation}
Finally, let $\bar \mu_\lambda=\pi_*\mu_\lambda$, supported on $\mathcal{C}$.

\subsubsection*{Conditions on $(n_m)$}

Here we collect the choices from earlier, and set all further conditions on the sequence $(n_m)$. Choose $n_1$ to be so large that $\ell_1(n_1)>N(z)$. Assume that for all $m\in\N$, $n_m>n_{m-1}+\ell_2(n_{m-1}, \lambda)$. Further, always let $n_m$ be so large that for any $\bq\in \pi^{-1}(\mathcal{C})$ with $|\bq|\leq r_{-1}$, the impact of levels before $n_{m-1}+\ell_2(n_{m-1}, \lambda)$ is negligible: In particular we can choose for any fixed $\epsilon>0$, for $m$ sufficiently large and  $r_{-1}(n_m)\le |\bx|\le n_m$, that
\begin{equation}\label{eq:choice}
 \mu_\lambda[\bx]\le (\tfrac 12)^{|\bx|}2^{\epsilon r_{-1}}.
\end{equation}
We also assume that $r_1(n_m)>\ell_2(n_{m}, \lambda_0)$, where $\ell_2(n_m, \lambda)$ is the smallest integer satisfying
\[
\lambda_0^{\ell_2(n_{m}, \lambda_0)}\le\gamma^{n_{m}}. 
\]
Recall that $r_1(n_m)$ is the smallest integer satisfying $(\tfrac 12)^{r_1(n_m)}\le \lambda_1^{n_m}$.

\subsection{Upper bounds to measures of balls, and the proof of Theorem \ref{thm:main} (3)} \label{section73}

The aim of this subsection is to prove Theorem \ref{thm:main} (3). In the proofs of Theorem \ref{thm:main} (1) and (2) we have made use of the local dimension result given in part 1 of Lemma \ref{massdistribution}. An alternative route to dimension lower bounds is the potential theoretic method given in part 2 of Lemma \ref{massdistribution}. In this section we will make use of this second method.

Assume $\lambda_0, \lambda_1$ fixed. We cannot prove directly that the $t$-energy of the measure $\bar\mu_\lambda$ would be bounded for $\mathcal L$-almost all $\lambda$. The majority of this section is dedicated to the proof of a preliminary result, that for some sets $\Lambda_\epsilon\subset [\lambda_0, \lambda_1]$ with  $\mathcal L (\Lambda_\epsilon)>\lambda_1-\lambda_0-\epsilon$, for $\mathcal L$-almost every $\lambda\in \Lambda_\epsilon$, the $t$-energy is finite. We then have by part 2 of Lemma \ref{massdistribution} that by taking $\epsilon\to 0$ along a countable sequence we obtain 
$$\mathcal{L}(\{\lambda\in [\lambda_0,\lambda_1]\mid\dim \mathcal{C}< t\})=0.$$
This is the content of Lemma \ref{lowerboundmain} underneath. 
It is explained at the very end of the section, how the lower bound in case (3) of Theorem \ref{thm:main} follows from Lemma \ref{lowerboundmain}.

To be more precise, we now aim to prove that for some $t$ close to the dimension value from case (3) of Theorem \ref{thm:main}, 
\begin{equation}\label{eq:Lambda}
\int_{\Lambda_\epsilon}\iint|\pi(\bi) - \pi(\bj)|^{-t}\d\mu_{\lambda}(\bi) \d\mu_{\lambda}(\bj) \d\mathcal L<\infty. 
\end{equation}
This implies that $\iint|\pi(\bi) - \pi(\bj)|^{-t} \d\mu_{\lambda}(\bi) \d\mu_{\lambda}(\bj)<\infty $ for $\mathcal L$-a.e. $\lambda \in \Lambda_\epsilon$. The choice of $\Lambda_\epsilon$ is as follows: Let $\delta>0$. Let $\epsilon>0$ and denote by $\Lambda_\epsilon\subset [\lambda_0, \lambda_1]$ with $\mathcal L (\Lambda_\epsilon)>\lambda_1-\lambda_0-\epsilon$ such that statements of Lemma \ref{lem:lowerboundbranching}, and Lemma \ref{lem:unifpablo} simultaneously hold for all $\lambda\in \Lambda_\epsilon$ (note that $\delta>0$ plays a role in these lemmas).

We will now go through a series of manipulations of the problem to get to a point where we can begin handling the local behaviour of $\bar \mu_\lambda$ for different scales $(\tfrac 12)^r$ in lemmas underneath.

We will be using $C$ from Lemma \ref{agreement}. Note that since $\frac{1}{2\gamma}>\frac{1}{2}$ we will have $\lambda_0>\frac{1}{2}$ and so this $C$ can be taken uniformly across the interval $[\lambda_0,\lambda_1]$. Let 
$$A_r=\{(\lambda,\bi, \bj)\in \Lambda_\epsilon\times \mathcal{C}\times \mathcal{C}\mid |\pi(\bi) - \pi(\bj)|\in [C(\tfrac 12)^r,C(\tfrac 12)^{r+1})\},$$
which we will use to split the integral from \eqref{eq:Lambda} into level sets. This gives the following upper bound: 
\begin{equation}\label{eq:energystart}
\begin{aligned}
\int_{\Lambda_\epsilon}\iint|\pi(\bi) - \pi(\bj)|^{-t} \d\mu_{\lambda}(\bi) \d\mu_{\lambda}(\bj) \d\mathcal L\\
&\le C^{-t}\sum_{r=1}^\infty 2^{rt}\iiint \chi_{A_r}  \d\mu(\bi) \d\mu(\bj) \d\mathcal L.
\end{aligned}
\end{equation}
Denote 
 \[
M(r):=\iiint \chi_{A_r}  \d\mu_{\lambda}(\bi) \d\mu_{\lambda}
(\bj)\, d\mathcal L.
 \]
 In order to achieve the estimate \eqref{eq:Lambda}, we need to bound this quantity from above for all sufficiently large values of $r$. It should be noted that at this point is not possible to use Fubini's Theorem as the measure $\mu_{\lambda}$ depends on $\lambda$.  
 
 We are now in a position to start studying different values of $r$, with the aim of obtaining good enough bounds to $M(r)$. Ultimately, whether the bounds below are `good enough' is a matter of how large can we make $t$ and still obtain \eqref{eq:Lambda}. It is justified at the end of the section in the proofs of Lemma \ref{lowerboundmain} and the proof of Theorem \ref{thm:main}  as its consequence that $t$ can be taken to be close to the dimension value from case (3) of Theorem \ref{thm:main}. 
 
Let $\epsilon>0$ be fixed throughout the following series of lemmas.  Let $r\in \N$ and fix $m\in \N$ to be the value for which $r_{-1}=r_{-1}(n_m)\leq r \leq r_1=r_1(n_m)$.   Note that since we assume $\lambda>\tfrac 1{2\gamma}$, by definitions of $r_0$ and $\ell_2(n_{m})$, we have $r_0<n_{m}$. Our first region of values considered is:
\begin{lemma}\label{lem:mainenergy}
Assume that $r_{-1}\leq r<r_0<n_m$ is sufficiently large. We have 
$$M(r)\lesssim 2^{r\epsilon}(2^{-n_m}+2^{-2r}\lambda_0^{-r}).$$
\end{lemma}

\begin{proof} 

Notice that by Lemma \ref{agreement} $|\bi\wedge \bj|\ge r$ whenever  $|\pi(\bi) - \pi(\bj)| \leq C(\tfrac 12)^r$. The rest of the argument does not depend on the $y$-coordinate. We also are bounding the integrals from above and so we can now remove the lower bounds in the level sets of the integral, and set
\[
A'(r):=\{(\lambda,\bi, \bj)\in \Lambda_\epsilon\times \mathcal{C}\times \mathcal{C}\mid |\sum_{\ell=r}^\infty (i_\ell - j_\ell)\lambda^\ell|<(\tfrac 12)^r\} 
\]
and the integral
 \begin{equation}\label{eq:Mdashes}
M'(r):=\iiint \chi_{A'(r)}  \d\mu_{\lambda}(\bi) \d\mu_{\lambda}(\bj) \d\mathcal L.
 \end{equation}

Recall that $k(r)$ is the number which satisfies $\lambda ^{k(r)}=(\tfrac 12)^r$. Notice that since $r<r_0$, $k(r)< n_m+\ell_2(n_m,\lambda_0)$. Hence, up to an absolute constant, the tail of the series beyond $n_m+\ell_2(n_m, \lambda)$ makes no difference to the estimate. Furthermore, since both $\bi$ and $\bj\in \mathcal{C}$, they both hit the target at time $n_m$ and, again up to absolute constants, 
\[
\left|\sum_{\ell=n_m}^{n_m+\ell_2(n_m,\lambda)} (i_\ell - j_\ell)\lambda^\ell\right|\lesssim\lambda^{n_m+\ell_2(n_m, \lambda)}< (\tfrac 12)^r. 
\]
Hence the series degenerates: that is 
\begin{equation}\label{eq:series}
|\sum_{\ell=r}^\infty (i_\ell - j_\ell)\lambda^\ell|<(\tfrac 12)^r
\end{equation}
requires that 
\begin{equation}\label{eq:sumest}
|\sum_{\ell=r}^{n_m} (i_\ell - j_\ell)\lambda^\ell|\lesssim(\tfrac 12)^r. 
\end{equation}
It will be helpful to write for $w<n_m$
$$A'(r,w)=A'(r)\cap\{(\bi,\bj,\lambda)\mid|\bi\wedge \bj|=w\}$$
and
$$A'(r,n_m)=A'(r)\cap\{(\bi,\bj,\lambda)\mid|\bi\wedge \bj|\geq n_m\}.$$
We let 
\begin{equation}\label{eq:Mdashesw}
M'(r,w)=\iiint \chi_{A'(r,w)}  \d\mu_{\lambda}(\bi) \d\mu_{\lambda}(\bj) \d\mathcal L.
\end{equation}

Note that 
\[
M'(r)=\sum_{w=r}^{n_m-1} M'(r, w)+M'(r, n_m). 
\]
We start by estimating $M'(r,n_m)$. In this case we automatically have that for any $\lambda\in \Lambda_{\epsilon}$
\[
|\sum_{\ell=r}^\infty (i_\ell - j_\ell)\lambda^\ell|\lesssim (\tfrac 12)^r. 
\]
We then have, by the definition of $\mu$ and \eqref{eq:choice} 
\begin{eqnarray*}
M'(r,n_m)&\lesssim &\mu\times\mu(\{(\bi,\bj)\mid|\bi\wedge \bj|\geq n_m\}\\
&\lesssim &2^{-n_m}2^{r_{-1}\epsilon}\leq 2^{-n_m}2^{r\epsilon}.
\end{eqnarray*} 

We now turn our attention to the estimation $M'(r,w)$ when $w<n_m$. In this case, by \eqref{eq:sumest} the condition 
\[
|\sum_{\ell=r}^\infty (i_\ell - j_\ell)\lambda^\ell|<(\tfrac 12)^r
\]
reduces to 
\[
|\sum_{\ell=w+1}^{n_m} (i_\ell - j_\ell)\lambda^\ell|\lesssim (\tfrac 12)^r. 
\]
For convenience, denote 
\[
P(w, n_m):=\{((i_{w+1}, \dots, i_{n_m}), (j_{w+1}, \dots, j_{n_m}))\mid i_{w+1}\neq j_{w+1}\}. 
\]
Recall that between times $r$ and $n_m$, the measure $\mu_\lambda$ is split between cylinders equally. At time $w>r>r_1$, we also know by \eqref{eq:choice} that measure of any cylinder of length $k$ is equal to $(\tfrac 12)^k$, up to a small error. We thus have, using Theorem \ref{thm:transversality},
\begin{eqnarray*}
&M'(r,w)\\
&\lesssim&2^{-w+r\epsilon}\sum_{P(w,n_m)}2^{-2(n_m-w)+1}\int \chi_{\{\lambda\mid|\sum_{\ell=w+1}^{n_m} (i_\ell - j_\ell)\lambda^\ell|\lesssim (\tfrac 12)^r\}}\d\mathcal L\\
&\lesssim&  2^{-w}2^{-r(1-\epsilon)}\lambda_0^{-w}\\
&\lesssim&   2^{(-2+\epsilon)r} \lambda_0^{-r}(2\lambda_0)^{-(w-r)}.
\end{eqnarray*}
Putting these estimates together, and using that $(2\lambda_0)>1$, we get that 
$$M'(r)\lesssim 2^{r\epsilon}(2^{-n_m}+2^{-2r}\lambda_0^{-r}).$$
\end{proof}

\begin{lemma}\label{lem:r0rn}
For $m$ sufficiently large, for $r_0\leq r<n_m$, we have
$$M(r)\lesssim 2^{2r\epsilon}((\tfrac 12)^{2r}\lambda_0^{-r}\left(\frac{\lambda_1}{\lambda_0}\right)^{n_m} + \lambda_0^{-n_m}\gamma^{-n_m}(\tfrac 12)^{r+n_m}).$$
\end{lemma}
\begin{proof}

We use the reduction of $M(r)$ to $M'(r)$ as in \eqref{eq:Mdashes}, and adopt a similar approach as in the previous lemma, but we now need to split the series from \eqref{eq:series} into three parts. We use

\begin{equation}\label{eq:splitsum}
\sum_{\ell=r}^\infty (i_\ell - j_\ell)\lambda^\ell = \sum_{\ell=r}^{n_m} (i_\ell - j_\ell)\lambda^\ell +\sum_{\ell=n_m}^{n+\ell_2(n_m,\lambda)} (i_\ell - j_\ell)\lambda^\ell +\sum_{\ell=n+\ell_2(n_m,\lambda)}^\infty (i_\ell - j_\ell)\lambda^\ell. 
\end{equation}
For the time being, consider $\bi, \lambda$ and $\bj|_{n_m+\ell_2(n_m,\lambda)}$ fixed. Splitting the last term of \eqref{eq:splitsum} into a difference of the series for $i_\ell$ and $j_\ell$, we will first look for an estimate for the measure of those $j_{n_m+\ell_2(n_m,\lambda)}, \dots$ which satisfy
\begin{equation}\label{eq:condition}
|\sum_{\ell=r}^{n_m+\ell_2(n_m,\lambda)}(i_\ell - j_\ell)\lambda^\ell +\sum_{\ell=n_m+\ell_2(n_m,\lambda)}^\infty i_\ell\lambda^\ell - \sum_{\ell=n_m+\ell_2(n_m,\lambda)}^\infty j_\ell\lambda^\ell|< C(\tfrac 12)^r. 
\end{equation}
Denote $x:=\sum_{\ell=r}^{n_m+\ell_2(n_M,\lambda)} (i_\ell - j_\ell)\lambda^\ell +\sum_{\ell=n_m+\ell_2(n_m,\lambda)}^\infty i_\ell\lambda^\ell$ and notice that the above inequality implies 
\[
\sum_{\ell=n_m+\ell_2(n_m,\lambda)}^\infty j_\ell\lambda^\ell\in B(x,  C(\tfrac 12)^r). 
\]
Recall the definition of $k(r)$ as the first integer for which $\lambda^{k(r)}\le C(\tfrac 12)^r$. Hence, it only adds a multiplicative constant to the estimate to curtail the sum on the left at $k(r)$. Thus using the condition from Lemma \ref{lem:unifpablo} combined with the estimate from Corollary \ref{cor:branching} gives

\begin{equation}\label{eq:muest1}
\begin{aligned}
\mu_{\lambda}&\{(j_{n_m+\ell_2}, \dots )\mid \sum_{\ell=n_m+\ell_2(n_m,\lambda)}^{k(r)} j_\ell\lambda^\ell\in B(x,  C(\tfrac 12)^r)\}\\
&\leq 2^{-k(r)+n_m+\ell_2(n_m,\lambda)}\#\{(j_{n_{m}+\ell_2(n_m,\lambda)+1},\ldots,j_{k(r)})\mid \sum_{\ell=n_m+\ell_2(n_m,\lambda)}^{k(r)} j_\ell\lambda^\ell\in B(x,  C(\tfrac 12)^r)\}\\
&\lesssim \lambda^{(k(r)-n_m-\ell_2(n_m,\lambda))(1-\epsilon)}\\
&\lesssim 2^{-r(1-\epsilon)}\lambda^{-n_m-\ell_2(n_m,\lambda)}\\
&\lesssim 2^{-r(1-\epsilon)}\gamma^{-n_m}\lambda_0^{-n_m}
\end{aligned}
\end{equation}
Use the notation $M'(r, w)$ and $M'(r, n_m)$ as in \eqref{eq:Mdashesw} in the proof of Lemma \ref{lem:mainenergy}, obtained by splitting the sets in $M'(r)$ from \eqref{eq:Mdashes} according to $w=|\bi\wedge \bj|$. In the case when $w= n_m$ we get the following estimate 
\begin{equation}\label{eq:Mrn}
M'(r,n_m)\lesssim \lambda_0^{-n_m}\gamma^{-n_m}(\tfrac 12)^{r}2^{2r\epsilon}(\tfrac 12)^{n_m}.
\end{equation}
This is a consequence of \eqref{eq:muest1}, with the additional term $(\tfrac 12)^{n_m}2^{r\epsilon}$ coming from the condition $|\bi\wedge \bj|\ge n_m$. 

We now consider the case when $w<n_m$. Note that if 
\[
|\sum_{\ell=r}^{n_m} (i_\ell - j_\ell)\lambda^\ell |=|\sum_{\ell=w}^{n_m} (i_\ell - j_\ell)\lambda^\ell |\gtrsim\gamma^{n_m}\lambda^{n_m}, 
\]
for some large enough constant, then there is no chance of the total series summing up to be less than $(\tfrac 12)^r$, since the other two terms in \eqref{eq:splitsum} are both bounded from above by $\gamma^{n_m}\lambda^{n_m}$. 

Recall that $\mu_\lambda$ is a product measure, with the definition up to $n_m$ independent of what happens after. From these observations, we obtain
\begin{align*}
M'(r,w)&=\iiint\chi_{A'(r,w)} \chi_{\{(\bi,\bj,\lambda)\mid|\sum_{\ell=w+1}^\infty (i_{\ell}-j_{\ell})\lambda^l|<2^{-r}\}}\d\mu_{\lambda} \d\mu_{\lambda} \d\mathcal{L}\\
&\lesssim \iiint\chi_{A'(r,w)}\chi_{\{(\bi,\bj,\lambda)\mid|\sum_{\ell=w+1}^{n_m} (i_{\ell}-j_{\ell})\lambda^l|\lesssim \lambda^{n_m}\gamma^{n_m}\}}\cdot\\
&\quad \quad\cdot\chi_{\{(\bi,\bj,\lambda)\mid|\sum_{\ell=n_m+\ell_2(n_m,\lambda)+1}^\infty (i_{\ell}-j_{\ell})\lambda^l|\leq 2^{-r}\}} \d\mu_{\lambda} \d\mu_{\lambda} \d\mathcal{L}\\
&\lesssim \iiint\chi_{A'(r,w)}\chi_{\{(\bi,\bj,\lambda)\mid|\sum_{\ell=w+1}^{n_m} (i_{\ell}-j_{\ell})\lambda^l|\lesssim \lambda^{n_m}\gamma^{n_m}\}} \d\mu_{\lambda} \d\mu_{\lambda} \cdot \\
&\quad\quad \cdot \iint \chi_{\{(\bi,\bj,\lambda)\mid|\sum_{\ell=n_m+\ell_2(n_m,\lambda)+1}^\infty (i_{\ell}-j_{\ell})\lambda^l|\leq 2^{-r}\}} \d\mu_{\lambda} \d\mu_{\lambda} \d\mathcal{L} =: \int I_1 \cdot I_2 \d\mathcal L. 
\end{align*}
For the second integral $I_2 $ we obtain from \eqref{eq:muest1}
\begin{equation}
\label{eq:smallw2}
\iint \chi_{\{(\bi,\bj,\lambda)\mid|\sum_{\ell=n_m+\ell_2(n_m,\lambda)+1}^\infty (i_{\ell}-j_{\ell})\lambda^l|\leq 2^{-r}\}} \d\mu_{\lambda} \d\mu_{\lambda} \lesssim \lambda_0^{-n_m} \gamma^{-n_m}(\tfrac 12)^{r}2^{r\epsilon},
\end{equation}
which is uniform in $\lambda$. Hence, 
\[
\int I_1 I_2 \d\mathcal L\lesssim \lambda_0^{-n_m} \gamma^{-n_m}(\tfrac 12)^{r}2^{r\epsilon}\int I_1 \d\mathcal L.  
\]

In the remaining integral, we split according to the possible length $n_m$ beginnings of $\bi$ and $\bj$, recalling that by the conditions in $A'(r, w)$, $|\bi\wedge \bj|=w$. Hence, by the estimate \eqref{eq:choice},
\begin{align*}
\iiint&\chi_{A'(r,w)}\chi_{\{(\bi,\bj,\lambda)\mid|\sum_{\ell=w+1}^{n_m} (i_{\ell}-j_{\ell})\lambda^l|\lesssim \lambda^{n_m}\gamma^{n_m}\}} \d\mu  \d\mu\d\mathcal L \\
&\lesssim \sum_{|\bi|=n_m}\sum_{|\bj|=n_m}\chi_{\{(\bi, \bj,\lambda)\mid |\bi\wedge \bj|=w\text{ and }|\sum_{\ell=w+1}^{n_m} (i_{\ell}-j_{\ell})\lambda^l|\lesssim \lambda^{n_m}\gamma^{n_m}\}} \\
&\lesssim 2^{r\epsilon}\sum_{|\bi|=n_m}\sum_{|\bj|=n_m}(\tfrac 12)^{2n_m} \int \chi_{\{\{(\bi, \bj,\lambda)\mid |\bi\wedge \bj|=w\text{ and }|\sum_{\ell=w+1}^{n_m} (i_{\ell}-j_{\ell})\lambda^l|\lesssim \lambda^{n_m}\gamma^{n_m}\}} \d \mathcal L. 
\end{align*}
In the remaining integral we can now apply transversality (Theorem \ref{thm:transversality}) to estimate, for any fixed words $|\bi|=n_m$ and $|\bj|=n_m$ with $|\bi\wedge\bj|=w$, $r\le w\le n_m$, 
\begin{equation}\label{eq:muest2}
\begin{aligned}
\int\chi_{\{(\bi,\bj,\lambda)\mid|\sum_{\ell=w+1}^{n_m} (i_{\ell}-j_{\ell})\lambda^l|\lesssim \lambda^{n_m}\gamma^{n_m}, |\bi\wedge \bj|=w\}}\d\mathcal L &\le 
\mathcal L\{\lambda\mid |\sum_{\ell=w+1}^{n_m} (i_\ell - j_\ell)\lambda^\ell |\lesssim \gamma^{n_m}\lambda^{n_m}\}\\
&\lesssim \lambda_1^{n_m-w}\gamma^{n_m} 
\end{aligned}
\end{equation}
where we are using that $\lambda^{n_m-w}\leq\lambda_1^{n_m-w}$.
Hence, adding up \eqref{eq:muest2} over the possible $\bi, \bj$ as above, we obtain as an upper bound to $\int I_1 \d \mathcal L$
\begin{equation}\label{eq:bigw2}
\iiint\chi_{\{(\bi,\bj,\lambda)\mid|\sum_{\ell=w+1}^{n_m} (i_{\ell}-j_{\ell})\lambda^l|\lesssim \lambda^{n_m}\gamma^{n_m}, |\bi\wedge \bj|=w\}}\d\mathcal L  \d\nu  \d\nu \lesssim 2^{r\epsilon}\lambda_1^{n_m-w}\gamma^{n_m}(\tfrac 12)^w. 
\end{equation}

The estimate for $M$ now consists of these pieces, \eqref{eq:smallw2}, and \eqref{eq:bigw2} which, summing up over $w$,  combine to give:
\[
\sum_w M(r, w)\lesssim \sum_{w=r}^{n_m} 2^{2r\epsilon}(\tfrac 12)^{r+w}\lambda_0^{-w}\left(\frac{\lambda_1}{\lambda_0}\right)^{n_m}.
\]
Summing up the above and \eqref{eq:Mrn} gives
\[
M(r)\lesssim 2^{2r\epsilon}\left((\tfrac 12)^{2r}\lambda_0^{-r}\left(\frac{\lambda_1}{\lambda_0}\right)^{n_m}+\lambda_0^{-n}\gamma^{-n_m}(\tfrac 12)^{r+n_m}\right) 
\]
as desired. 
\end{proof}

In the following few lemmas, we will momentarily abandon the energy integral and the corresponding product measure, and instead we investigate the measure $\overline{\mu}_{\lambda}(Q(x,C(\tfrac 12)^r))$ for some fixed $\lambda \in \Lambda_\epsilon$. Good upper bounds $\overline{\mu}_{\lambda}(Q(x, C(\tfrac 12)^r)\lesssim b(r)$ will imply bounds for $M(r)$, as
\begin{equation}\label{eq:fixedmsr}
M(r)\lesssim \mathcal L(\Lambda_\epsilon)\mu(\mathcal{C})b(r). 
\end{equation}

\begin{lemma}\label{lem:nrn+ell_1}
Let $\lambda \in \Lambda_\epsilon$. Let $x\in \mathcal{C}$ and $n_m<r<n_m+\ell_1(n_m)$. Then 
\[
\overline{\mu}_{\lambda}((Q(x,C(\tfrac 12)^r)\lesssim (\tfrac 12)^{n_m+r}\lambda^{-n_m-\ell_2(n_m,\lambda)}2^{r\epsilon}.
\]
\end{lemma}
\begin{proof}
This case is in a sense analogous to Lemma \ref{lma:r0rnl}. Notice first of all that since $r_0<n_m<r<n_m+\ell_1(n_m)$, the cube $Q(x, C(\tfrac 12)^r)$ intersects only one cylinder $\pi[\bx|_{n_m+\ell_1(n_m)}]$ from level $n_m+\ell_1(n_m)$, but it does not contain it. Hence, the question is, how much of the mass of this cylinder $\bar\mu_\lambda[\pi[\bx|_{n_m+\ell_1(n_m)}]]\lesssim (\tfrac 12)^{n_m}$ is within the ball. 

Within this cylinder the mass is distributed between cylinders from level $n_m+\ell_2(n_m,\lambda)$, which is the ambiguous region for distributing mass in the definition of $\mu_\lambda$. Each of the projected cylinders $\pi[\bj_{n_m+\ell_2(n_m, \lambda)}]$ is of width $\lambda^{n_m+\ell_2(n_m, \lambda)}\gtrsim C(\tfrac 12)^r$. Hence at worst, $Q(x, C(\tfrac 12)^r)$ intersects all of these cylinders, but it does not contain them either. 

Recall that ultimately, the behaviour of the cylinders beyond level $k(r)$ makes no difference to the measure estimates, so let us estimate the cylinders from level $k(r)$. Within those level $n_m+\ell_2(n_m, \lambda)$ cylinders intersecting $\pi^{-1}(Q(x,C2^{-r} ))$, the number of level $k(r)$ cylinders which  intersect $\pi^{-1}(Q(x,C2^{-r} ))$ is,  up to  a constant, $(2\lambda^{1-\epsilon})^{k(r)-n_m-\ell_2(\lambda,n_m)}$. This estimate follows from the condition in Lemma \ref{lem:unifpablo} combined with the estimate in Corollary \ref{cor:branching}. This gives
$$\overline{\mu}_{\lambda}(\pi[\bj|_{n_m+\ell_2(n_m,\lambda)}]\cap Q(x, C(\tfrac 12)^r))
\lesssim 2^{r\epsilon}(\tfrac 12)^{r}\lambda^{(-n_m-\ell_2(n_m,\lambda))}.
$$
 Hence, we finally obtain
\[
\begin{aligned}
\overline{\mu}_{\lambda}(Q(x, C(\tfrac 12)^r))&\le \sum_{|\bj|=n_m+\ell_2(n_m,\lambda): \pi[\bj]\cap Q(x, C(\tfrac 12)^r)\neq \emptyset}\mu(\pi[\bj]\cap Q(x, C(\tfrac 12)^r))\\
&\lesssim (\tfrac 12)^{n_m+r(1-\epsilon)}\lambda^{(-n_m-\ell_2(n_m, \lambda))(1-\epsilon)}\lesssim 2^{r\epsilon}(\tfrac 12)^{n_m}\lambda^{(-{n_m}-\ell_2(n_m,\lambda))}
\end{aligned}
\]

as desired. 
\end{proof}

\begin{lemma}\label{lm:n+ell_1nn+ell_2}
Let $\lambda\in \Lambda_\epsilon$. Let $x=(x_1,x_2)\in \mathcal{C}$ and $n_m+\ell_1(n_m)<r<n_m+\ell_2(n_m,\lambda)$. Then 
\[
\overline{\mu}_{\lambda}(Q(x, C(\tfrac 12)^r))\lesssim 2^{2r\epsilon}\lambda^{-\epsilon\xi n_m}\lambda^{-r-\xi n_m}2^{-2r+\ell_2(n_m,\lambda)-\xi n_m}.
\]
\end{lemma}
\begin{proof}
Let $\pi(\bx)=x$. As before, we know that $Q(x, C2^{-r})$ only intersects one cylinder from level $r$. Hence, we first want to estimate the measure of the cylinder $\mu[\bx|_r]$. Notice that 
\begin{align*}
\mu_\lambda[\bx|_r]=\sum\mu_\lambda[\bj], 
\end{align*}
where the sum is over those cylinders $[\bj]$ with $|\bj|=\ell_2(n_m,\lambda)+n_m$, 
$\bj|_r=\bx|_r$ and $\pi([\bj])\cap \mathcal{C} \neq \emptyset$. The last of 
these conditions implies that $\pi[\bj]\cap B(\pi(\bx|_{n_m}\bz), 
\gamma^{n_m}\lambda^{n_m})\neq \emptyset$, which in turn implies that $
\pi([\sigma^{n_m}(\bj)])\cap B(\pi(\sigma^{n_m}(\bx)), 2\gamma^{n_m}\lambda^{n_m})$ 
since $\pi(\sigma^{n_m}(\bx))\in B(z, \gamma^{n_m})$. We now estimate the 
number of such $\bj$. Here, since $\lambda^{\ell_2(n_m,\lambda)}\approx\gamma^n$
\begin{align*}
\{|\bj|=n_m+\ell_2&(n_m, \lambda)\mid \bj|_r=\bx|_r, \pi[\sigma^{n_m}\bj]\cap B(\pi(\sigma^{n_m}(\bx)), 2\gamma^{n_m})\neq \emptyset\}\\
&=\{|\bj|=n_m+\ell_2(n_m, \lambda)\mid |\sum_{i=r}^{n_m+\ell_2(n_m, \lambda)}(x_{i} - j_i)\lambda^i|<2\lambda^{n_m+\ell_2(n_m, \lambda)}\}\\
&=\{|\bj|=n_m+\ell_2(n_m, \lambda)\mid |\sum_{i=1}^{n_m+\ell_2(n_m, \lambda)-r}(x_{i+r} - j_{i+r})\lambda^i|<2\lambda^{n_m+\ell_2(n_m, \lambda)-r}\}. 
\end{align*}
This condition is equivalent to asking, how many words of length $|\ba|=n_m+\ell_2(n_m, \lambda)-r$ are there, such that for some infinite word $\bi$ in $[\ba]$, we have $\pi(\bi)\in B(\sigma^r(\bx), 2\lambda^{n_m+\ell_2(n_m, \lambda)-r})$. As before by combining the condition in Lemma \ref{lem:unifpablo} with the technique from Corollary \ref{cor:branching}, we can bound this from above by a constant multiple of $(2\lambda)^{(r-n_m+\ell_2(n_m,\lambda))(1-\epsilon)}$. 

By \eqref{eq:measestimate} for each $|\bj|=n_m+\ell_2(n_m, \lambda)$ in the support of $\mu_\lambda$ we have,
\[
\mu_{\lambda}[\bj]\lesssim 2^{r\epsilon}2^{-{n_m}}(2\lambda)^{-\xi n_m}\lambda^{-\epsilon \xi n_m}.
\]
Hence, 
\[
\mu[\bx|_r]= \sum\mu[\bj]\lesssim 2^{r\epsilon}(2\lambda)^{(n_m+\ell_2(n_m,\lambda)-r)(1-\epsilon)}2^{-n_m}2^{r\epsilon}(2\lambda)^{-\xi n_m}\lambda^{-\epsilon \xi n_m}. 
\]

The $\mu$-measure inside this cylinder $[\bx|_r]$ is divided evenly between its sub-cylinders from level $n_m+\ell_2(n_m,\lambda)$. Since $r<n_m+\ell_2(n_m,\lambda)$, these sub-cylinders are not contained in the ball $Q(x, C(\tfrac 12)^r)$. However, using Lemma \ref{thm:pablo} we know for each $|\bj|=n_m+\ell_2(n_m,\lambda)$ that   
\begin{align*}
\#\{\omega\in\{0,1\}^{k-n_m-\ell_2(n_m,\lambda)}\mid&\sum_{i=1}^{k-n_m-\ell_2(n_m,\lambda)}\omega_i\lambda^i\in B(x_1-\sum_{i=1}^{n_m+\ell_2(n_m,\lambda)}j_i\lambda_i,C2^{-r})\}\\
&\lesssim (2\lambda^{1-\epsilon})^{k(r)-n_m+\ell_2(n_m,\lambda)}.
\end{align*}
This gives that for each $|\bj|=n_m+\ell_2(n_m,\lambda)$
$$\mu_{\lambda}(\{\omega\mid\omega\in [\bj]\cap Q(x,C2^{-r})\})\lesssim (\lambda^{1-\epsilon})^{{k(r)-n_m-\ell_2(n_m,\lambda)}}\leq (2^{-r}\lambda^{-\ell_2-n_m})^{1-\epsilon}.$$

Putting this together 
\begin{align*}
\overline{\mu}_{\lambda}(Q(x, C2^{-r}))&\lesssim (2^{-r}\lambda^{-\ell_2(n_m,\lambda)-n_m})^{1-\epsilon}2^{r\epsilon}(2\lambda)^{(n_m+\ell_2(n_m,\lambda)-r)(1-\epsilon)}2^{-n_m}2^{r\epsilon}(2\lambda)^{-\xi n_m}\lambda^{-\epsilon \xi n_m}\\
&\lesssim 2^{2r\epsilon}\lambda^{-\epsilon\xi n_m}\lambda^{-r-\xi n_m}2^{-2r+\ell_2(n_m,\lambda)-\xi n_m}.
\end{align*}
\end{proof}
We now are left with only the final case to consider. 
\begin{lemma}\label{lem:rn}
For $n_{m}+\ell_{2}(n_m, \lambda)\le r<r_1<n_{m+1}$. Then for $x\in\mathcal{C}$, 
\[
\overline{\mu}_{\lambda}(Q(x, C(\tfrac 12)^r)\lesssim  \lambda^{-\epsilon r} 2^{2r\epsilon} 2^{-2r+\ell_2(n_m,\lambda)-\xi n_m}\lambda^{-\xi n_m-r} 
\]
\end{lemma}
\begin{proof}
It follows from \eqref{eq:measestimate} that
\begin{eqnarray*}
\mu_{\lambda}[\bx|_{r}]&\lesssim& 2^{r\epsilon}2^{-r+\ell_2(n_m)}(2\lambda)^{-\xi n_m}\lambda^{-\epsilon \xi n_m}\\
&\lesssim&\lambda^{-\epsilon r} 2^{r\epsilon}2^{-r+\ell_2(n_m,\lambda)-\xi n_m}\lambda^{-\xi n_m}
\end{eqnarray*}

From here the argument is analogous to the argument in Lemma \ref{lem:lower31}, with the above estimate replacing \eqref{eq:eq1}. Hence we get that
\begin{eqnarray*}
\overline{\mu}_{\lambda}(Q(x, C(\tfrac 12)^r)&\lesssim&  \lambda^{-\epsilon r} 2^{r\epsilon}2^{-r+\ell_2(n_m,\lambda)-\xi n_m}\lambda^{-\xi n_m}\lambda^{(k(r)-r)(1-\epsilon)}\\
&\lesssim&  \lambda^{-\epsilon r} 2^{2r\epsilon} 2^{-2r+\ell_2(n_m,\lambda)+\xi n_m}\lambda^{-\xi n_m-r}
\end{eqnarray*}
which gives what we want.
\end{proof}

To complete the proof of Part (3) Theorem \ref{thm:main}, we need to show that for any $\gamma\in (0,1)$ for $\mathcal L$-almost every $\lambda\in [\tfrac 1{2\gamma}, \bar\lambda)$, for $\bar \nu_\lambda$-almost every $z\in F$, 
\[
\dimH R^*(z, \lambda)\geq \frac{2\log 2+\log \lambda}{\log (2/\gamma)}. 
\] 
We begin by collecting all the Lemmas in this section together to prove the following statement. 
Recall that we have set
$$\eta(\lambda_0,\lambda_1)=\left(\frac{\log \lambda_0}{\log \lambda_1}-1\right)+(s+1)\left(\frac{\log\gamma}{\log \lambda_1}-\frac{\log\gamma}{\log \lambda_0}\right)$$
\begin{lemma}\label{lowerboundmain}
Let $\gamma\in (0,1)$. There exists $C_0>0$ such that for all $\epsilon>0$ there is a set $\Lambda_\epsilon\subset [\lambda_0, \lambda_1]$ with $\mathcal L (\Lambda_\epsilon)\geq \lambda_1-\lambda_0-\epsilon$, such that for $\mathcal L$-almost every $\lambda\in \Lambda_\epsilon$
\[
\dimH R^*(\pi_{\lambda}{\omega}, \lambda)\geq\frac{2\log 2+\log \lambda_0}{\log (2/\gamma)}-C_0\epsilon-\eta(\lambda_0,\lambda_1). 
\] 
\end{lemma}

\begin{proof}
We follow the strategy described at the beginning of Section \ref{section73}, and aim for the energy bound \eqref{eq:Lambda}. We fix $\epsilon>0$ and take $\Lambda_\epsilon\subset [\lambda_0,\lambda_1]$ as specified at the start of Section \ref{section73}. We will use the estimate \eqref{eq:energystart} from above, which gives
\[
\int_{\Lambda_\epsilon}\iint|\pi(\bi) - \pi(\bj)|^{-t} \d\mu(\bi) \d\mu(\bj) \d\mathcal L\lesssim \sum_{r=1}^\infty 2^{rt}M(r). 
\]
Recall that  Lemmas \ref{lem:nrn+ell_1}, \ref{lem:rn} and \ref{lm:n+ell_1nn+ell_2} imply an upper bound to $M(r)$ via \eqref{eq:fixedmsr}. We are looking for a $t$ such that this sum converges, and thus it is sufficient to find $t$ such that for all large $r$
\[
\frac{\log M(r)}{\log (\tfrac 12)^r}>t. 
\]
Recall the notation
\[
s=\frac{2\log 2+\log \lambda_0}{\log (2/\gamma)}. 
\]
We will show that there is some $C_0>0$ such that for all $r$ sufficiently large
\begin{equation}\label{eq:dimcons}
\frac{\log M(r)}{\log (\tfrac 12)^r}>s-C_0\epsilon-\eta(\lambda_0,\lambda_1). 
\end{equation}
Notice that for $\lambda_0>1/(2 \gamma)$, 
\[
s<\min\left\{2+\frac{\log\lambda_0}{\log \tfrac12}, \frac{-\log 2}{\log (\gamma \lambda_1)}, 2+\frac{\log \lambda_0}{\log 2}+\frac{\log \gamma}{\log(\gamma\lambda_1)}\right\}. 
\] 
Let us now consider all possible values of $r$ in turn. We start with the case of Lemma \ref{lem:mainenergy} where $r_1<r<r_0<n_m$, and $M(r)\lesssim 2^{r\epsilon}(2^{-n_m}+2^{-2r}\lambda_0^{-r})$. Then one of these terms dominates and either, for large enough $r$, 
\[
\frac{\log M(r)}{\log (\tfrac 12)^r} \geq\frac{\log 2+\log 2^{-n_m}}{\log 2^{-r}}-\epsilon \ge \frac{\log 2+\log 2^{-n_m}}{\log 2^{-r_0}}-\epsilon\ge\frac {\log 2-\log 2}{\log (\lambda_0\gamma)}-\epsilon>s-\epsilon, 
\]
or 
\begin{equation}\label{eq:fulldim}
\frac{\log M(r)}{\log (\tfrac 12)^r} \geq \frac{\log 2+\log(2^{-2r}\lambda_0^{-r})}{\log 2^{-r}}-\epsilon=2+\frac{\log \lambda_0}{\log 2}-\frac{\log 2}{r\log 2}-\epsilon>s-\epsilon. 
\end{equation}

Now consider the case of Lemma \ref{lem:r0rn}, $r_0<r<n_m$ and $M(r)$ is bounded from above by 
\[
\lesssim 2^{2r\epsilon}\left(\frac{\lambda_1}{\lambda_0}\right)^{n_m}((\tfrac 12)^{2r}\lambda_0^{-r}+\lambda_0^{-n_m}\gamma^{-n_m}(\tfrac 12)^{r+n_m}). 
\]
We write ${\bf M}(r)=((\tfrac 12)^{2r}\lambda_0^{-r}+\lambda_0^{-n_m}\gamma^{-n_m}(\tfrac 12)^{r+n_m})$ for the main term, and for the error term use the notation ${\bf E}(r)=2^{2r\epsilon}\left(\frac{\lambda_1}{\lambda_0}\right)^{n_m}$ with $M(r)\lesssim {\bf M}(r){\bf E}(r)$.

Let us first consider the error term ${\bf E}(r)$. For this term for large enough $r$, 
\begin{eqnarray*}
\frac{\log({\bf E}(r))}{\log 2^{-r}}&\geq& -2\epsilon-\frac{n_m}{r}\frac{\frac{\log(\lambda_1)}{\log(\lambda_0)}}{\log 2}\\
&\geq&  -3\epsilon+\frac{\log 2}{\log \lambda_1}\frac{\frac{\log(\lambda_1)}{\log(\lambda_0)}}{\log 2}\\
&\geq&-3\epsilon-\left(\frac{\log\lambda_0}{\log\lambda_1}-1\right).
\end{eqnarray*}
As of the two main terms, one of them dominates in the upper bound; if it is the first term, we repeat the computation \eqref{eq:fulldim} to obtain ${\bf M}(r)\geq s$. If the second term is the larger one, we use the estimate $r<n_m+\ell_1(n_m)$, noting that $2\gamma\lambda>1$ to obtain the estimate
\[
\frac{\log {\bf M}(r)}{\log (\tfrac 12)^r} \geq \frac{\log 2+\log (\lambda_0\gamma)^{-n_m}(\tfrac 12)^{r+n_m}}{\log (\tfrac 12)^r}\ge 1+\frac{n_m\log (2\gamma\lambda_0)^{-1}}{(n_m+\ell_1)\log (\tfrac 12)}-\frac{1}{r}. 
\]
Here, for any $\delta>0$, for large enough $r$, 
\begin{equation}\label{eq:ell1}
\frac{\log \gamma}{\log \tfrac 12}-\delta<\frac{\ell_1(n_m)}{n_m}<\frac{\log \gamma}{\log \tfrac 12}+\delta.
\end{equation}
 Hence, for large enough $r$,
 \begin{equation}  
\begin{aligned}\label{eq:optimal}
\frac{-\log 2+\log (\lambda_0\gamma)^{-n_m}(\tfrac 12)^{r+n_m}}{\log (\tfrac 12)^r}&
\ge 1+\frac{n_m}{n_m+\ell_1(n_m)} + \frac{n_m}{n_m+\ell_1(n_m)}\cdot\frac{\log (\lambda_0 \gamma)^{-1}}{\log \tfrac 12}-\epsilon\\
&\ge \frac{2\log 2+\log \lambda_0}{\log (2/\gamma)}-\epsilon\\
&>s-\epsilon. 
\end{aligned}
\end{equation}
Combining with the error term we get that
$$\frac{\log M(r)}{-r \log 2}\geq s-4\epsilon-\left(\frac{\log\lambda_0}{\log\lambda_1}-1\right)\geq s-4\epsilon-\eta(\lambda_0,\lambda_1).$$

The next range of $r$ is $n_m<r<n_m+\ell_1(n_m)$, where by Lemma \ref{lem:nrn+ell_1}, 
\[
M(r)\lesssim 2^{2r\epsilon}(\lambda_0\gamma)^{-n_m}(\tfrac 12)^{r+n_m}.
\]
Again, we first estimate the impact of the error term
\[
\frac{2r\epsilon \log 2}{-r\log 2}= -2\epsilon.  
\]

We can now repeat the computation \eqref{eq:optimal}, to obtain the same bound for the main term, and hence, 
\[
\frac{\log M(r)}{\log (\tfrac 12)^r} \ge \frac{2 \log 2 +\log \lambda_0}{\log(2/\gamma)}- \epsilon>s- 3\epsilon. 
\]

For the next range, consider $n_m+\ell_1(n_m)<r<n_m+\ell_2(n_m,\lambda_0)$.  Let $\lambda\in \Lambda_\epsilon$ be arbitrary, and recall from Lemma \ref{lm:n+ell_1nn+ell_2} that for $n_m+\ell_1(n_m)<r<n(n_m)+\ell_2(n_m,\lambda)$ and any $x\in\mathcal{C}$ we have,   
\begin{eqnarray*}
\overline{\mu}_{\lambda}(Q(x, (\tfrac 12)^r))&\lesssim& 2^{-2r+\ell_2(\lambda, n_m)-\xi n_m}\lambda^{-r-\xi n_m}2^{\epsilon 2r}\lambda^{-\epsilon\xi n_m}\\
&\le& 2^{\epsilon 2r}\lambda_0^{-\epsilon\xi n_m} 2^{-2r+\ell_2(\lambda, n_m)-\xi n_m}\lambda_0^{-r-\xi n_m}\\
&\leq&2^{\epsilon 2r}\lambda_0^{-\epsilon\xi n}2^{\ell_2(n_m,\lambda)-\ell_2(n_m,\lambda_0)}2^{-2r+\ell_1(n_m)}\lambda_0^{-r-\ell_2(n_m,\lambda_0)+\ell_1(n_m)}\\
&\leq&2^{\epsilon 2r}\lambda_0^{-\epsilon\xi n}2^{\ell_2(n_m,\lambda_1)-\ell_2(n_m,\lambda_0)}2^{-2r+\ell_1(n_m)}\lambda_0^{-r-\ell_2(n_m,\lambda_0)+\ell_1(n_m)}.
\end{eqnarray*}
This is now independent of $\lambda$ and so we can estimate from \eqref{eq:fixedmsr}
$$M(r)\lesssim 2^{\epsilon 2r}\lambda_0^{-\epsilon\xi n}2^{\ell_2(n_m,\lambda_1)-\ell_2(n_m,\lambda_0)}2^{-2r+\ell_1(n_m)}\lambda_0^{-r-\ell_2(n_m,\lambda_0)+\ell_1(n_m)}.$$
Thus we can write $M(r)\lesssim {\bf E}(r){\bf M}(r)$ where 
$${\bf E}(r)=2^{\epsilon 2r}\lambda_0^{-\epsilon\xi n}2^{\ell_2(n_m,\lambda_1)-\ell_2(n_m,\lambda_0)}$$
and
$${\bf M}(r)=2^{-2r+\ell_1(n_m)}\lambda_0^{-r-\ell_2(n_m,\lambda_0)+\ell_1(n_m)}.$$
We first have for $r$ sufficiently large and using that $\lambda_0^{-\epsilon\xi n_m}\leq 2^{r\epsilon}$
\begin{eqnarray*}
\frac{\log {\bf E}(r)}{-r\log 2}&\geq& -3\epsilon-\frac{\ell_2(n_m,\lambda_1)-\ell_2(n_m,\lambda_0)}{r}\\
&\geq& -3\epsilon-\frac{\ell_2(n_m,\lambda_1)-\ell_2(n_m,\lambda_0)}{n_m}\\
&\geq& -4\epsilon-\left(\frac{\log\gamma}{\log \lambda_1}-\frac{\log\gamma}{\log \lambda_0}\right)
\end{eqnarray*}
 
For the main part ${\bf M}(r)$ we use that $2+\frac{\log \lambda_0}{\log \frac{1}{2}}>s$ to get
\begin{eqnarray*}
{\bf M}(r)&\lesssim&  2^{-2r+\ell_2(\lambda_0)-\xi n_m}\lambda_0^{-r-\xi n_m}\\
&\lesssim& 2^{-n_m-r}\lambda_0^{-n_m-\ell_2(n_m,\lambda_0)}2^{-r+n_m+\ell_2(\lambda_0)-\xi n_m}\lambda_0^{-r-\xi n_m+n_m+\ell_2(\lambda_0,n_m)}\\
&\lesssim& (2^{-n-\ell_1(n_m)}\lambda_0^{-n_m})^s2^{-r+n_m+\ell_1(n_m)}\lambda_0^{-r+\ell_1(n_m)+n_m}\\
&\lesssim&  (2^{-n_m-\ell_1(n_m)}\lambda_0^{-n_m})^s2^{-2(r-n_m-\ell_1(n_m))}\lambda_0^{n_m-r+\ell_1(n_m)}\\
&\lesssim&  (2^{-n_m-\ell_1(n_m)}\lambda_0^{-n_m})^s (2^{-r+n_m+\ell_1(n_m)})^{2+\frac{\log \lambda_0}{\log \frac{1}{2}}}\\
&\lesssim&  (2^{-n_m-\ell_1(n_m)}\lambda_0^{-n_m})^s(2^{-r+n_m+\ell_1(n_m)})^s\leq 2^{-rs}.
\end{eqnarray*}
Putting this together for large enough $r$ we have
$$\frac{\log M(r)}{\log r}> s-5\epsilon- \left(\frac{\log\gamma}{\log \lambda_1}-\frac{\log\gamma}{\log \lambda_0}\right)$$

On the other hand, if $n_m+\ell_2(n_m,\lambda_0)<r<n_m+\ell_2(n_m,\lambda_1)$, then we can use the estimate from Lemma \ref{lm:n+ell_1nn+ell_2} for cubes radius $2^{-n_m-\ell_2(\lambda_0,n_m)}$. This gives
\begin{eqnarray*}
\overline{\mu}_{\lambda}(Q(x, (\tfrac 12)^r))&\lesssim& 2^{   2(n_m+\ell_2(n_m,\lambda_0))\epsilon}\lambda^{-\epsilon\xi n_m}\lambda^{-n_m-\xi n_m}2^{-2(n_m+\ell_2(n_m,\lambda_0))+\ell_2(n_m,\lambda)-\xi n_m}\\
&\lesssim&2^{2r\epsilon}\lambda^{-\epsilon\xi n_m}\lambda^{-n_m-2\ell_2(n_m,\lambda_0)+\ell_1(n_m)}2^{-2n_m-2\ell_2(n_m,\lambda_0)+\ell_2(n_m,\lambda_1)+\ell_1(n_m)}\\
&\lesssim&2^{2r\epsilon}\lambda^{-\epsilon\xi n_m}2^{\ell_2(n_m,\lambda_1)-\ell_2(n_m,\lambda_0)}\lambda^{-n_m-2\ell_2(n_m,\lambda_0)+\ell_1(n_m)}2^{-2n_m-\ell_2(n_m,\lambda_0)+\ell_1(n_m)}.
\end{eqnarray*}
Hence for $r$ in this range we can estimate $M(r)\lesssim {\bf E}(r){\bf M}(r)$ where
$${\bf E}(r)\leq 2^{2r\epsilon}\lambda_0^{-\epsilon\xi n_m}2^{\ell_2(n_m,\lambda_1)-\ell_2(n_m,\lambda_0)}$$
and
$${\bf M}(r)\leq \lambda_0^{-n_m-2\ell_2(n_m,\lambda_0)+\ell_1(n_m)}2^{-2n_m-\ell_2(n_m,\lambda_0)+\ell_1(n_m)}.$$
For the error term, using that $\lambda_0^{-\epsilon\xi n_m}\leq 2^{r\epsilon}$, for sufficiently large $r$
\begin{eqnarray*}
\frac{\log {\bf E}(r)}{-r\log 2}&\geq & -3\epsilon-\frac{\ell_2(n_m,\lambda_1)-\ell_2(n_m,\lambda_0)}{r}\\
&\geq&-3\epsilon-\frac{\ell_2(n_m,\lambda_1)-\ell_2(n_m,\lambda_0)}{n_m}\\
&\geq& -4\epsilon- \left(\frac{\log\gamma}{\log \lambda_1}-\frac{\log\gamma}{\log \lambda_0}\right)
\end{eqnarray*}
With the main term ${\bf M}(r)$ we work as in the previous part to get
\begin{eqnarray*}
{\bf M}(r)&\leq&\lambda_0^{-n_m-2\ell_2(n_m,\lambda_0)+\ell_1(n_m)}2^{-2n_m-\ell_2(n_m,\lambda_0)+\ell_1(n_m)}\\
&\lesssim&  2^{-2n_m-\ell_1(n_m)}\lambda_0^{-n_m-\ell_2(n_m,\lambda_0)} 2^{-2(\ell_2(n_m,\lambda_0)-\ell_1(n_m))}\lambda_0^{-\ell_2(n_m,\lambda_0)+\ell_1(n_m)}\\
&\lesssim&  (2^{-n_m-\ell_1(n_m)})^s(2^{-(\ell_2(n_m,\lambda_0)-\ell_1(n_m))})^{2+\frac{\log \lambda_0}{\log 2}}\\
&\lesssim& 2^{-s(n_m-\ell_2(n_m,\lambda_0))}\\
&\lesssim& 2^{s(\ell_2(n_m,\lambda_1)-\ell_2(n_m,\lambda_0))}2^{-rs}.
\end{eqnarray*}
Hence, for $r$ in this range sufficiently large, we get
\begin{eqnarray*}
\frac{\log M(r)}{-2\log r}&\geq& s-\epsilon\left(3-\frac{\log\gamma}{\log 2}\right)-(s+1)\frac{\ell_2(n_m,\lambda_1)-\ell_1(n_m,\lambda_0)}{r}\\
&\geq&s-5\epsilon-(s+1)\left(\frac{\log\gamma}{\log \lambda_1}-\frac{\log\gamma}{\log \lambda_0}\right) 
\end{eqnarray*}

Finally we have the range $n_m+\ell_2(\lambda_1,n_m)<r<r_1$. Here by Lemma \ref{lem:rn} we get
\begin{eqnarray*}
M(r)&\lesssim&   \lambda^{-\epsilon r} 2^{2r\epsilon} 2^{-2r+\ell_2(n_m,\lambda)-\xi n_m}\lambda^{-\xi n_m-r}\\
&\lesssim&  2^{\ell_2(n_m,\lambda_1)-\ell_2(n_m, \lambda_0)}\lambda^{-\epsilon r} 2^{2r\epsilon}2^{-2r+\ell_1(n_m)}\lambda_0^{-\ell_2(n_m,\lambda_0)+\ell_1(n_m)-r}
\end{eqnarray*}
Writing ${\bf M}(r)=2^{-2r-\ell_1(n_m)}\lambda_0^{-\ell_2(n_m,\lambda_0)+\ell_1(n_m)}$ and ${\bf E}(r)=2^{\ell_2(n_m,\lambda_1)-\ell_2(n_m, \lambda_0)}\lambda^{-\epsilon r}$ we get $M(r)\lesssim {\bf M}(r){\bf E}(r)$. 
We have, using that $\lambda^{-\epsilon r}\leq 2^{r\epsilon}$
$$\frac{\log({\bf E}(r))}{\log 2^{-r}}\geq -3\epsilon-\left(\frac{\log\gamma}{\log \lambda_1}-\frac{\log\gamma}{\log \lambda_0}\right).$$
We know that $2+\frac{\log \lambda_0}{\log 2}>\frac{2\log 2+\log\lambda_0}{\log(2/\gamma)}=s$, 
$2^{-2n_m}\lambda_0^{-n_m}$ is approximately $(2^{-n_m-\ell_1(n_m)})^s$ and $\lambda_0^{\ell_2(n_m,\lambda_0)}$ is approximately $2^{-n_m-\ell_1(n_m)}$, and where we can choose $r$ so large that the error given by these approximations is as small as we wish. Hence, 
\begin{eqnarray*}
{\bf M}(r)&=&2^{-2r+\ell_1(n_m)}\lambda_0^{-\ell_2(n_m,\lambda_0)+\ell_1(n_m)-r}\\
&=&  2^{-2n_m-\ell_1(n_m)}\lambda_0^{-n_m-\ell_2(n_m,\lambda_0)}2^{-2(r-n_m-\ell_1(n_m))}\lambda_0^{-r+n_m+\ell_1(n_m)}\\
&=&   (2^{-n_m-\ell_1(n_m)})^s(2^{-r+\ell_1(n_m)+n_m})^{2+\frac{\log \lambda_0}{\log 2}}\\
&\leq& (2^{-n_m-\ell_1(n_m)})^s(2^{-r+\ell_1(n_m)+n_m})^{s}\leq 2^{-rs}
\end{eqnarray*}
Thus in this range we have that
$$\frac{\log M(r)}{\log r}\geq   s-3\epsilon-\left(\frac{\log\gamma}{\log \lambda_1}-\frac{\log\gamma}{\log \lambda_0}\right).$$

Putting everything together this completes the proof with $C_0=5$ and 
$$\eta(\lambda_0,\lambda_1)=\left(\frac{\log \lambda_0}{\log \lambda_1}-1\right)+(s+1)\left(\frac{\log\gamma}{\log \lambda_1}-\frac{\log\gamma}{\log \lambda_0}\right).$$
\end{proof}

\begin{proof}[Proof of Theorem \ref{thm:main} (3) as a consequence of Lemma \ref{lowerboundmain}]\label{proof}

Let us first check that for almost every $\lambda\in [\lambda_0, \lambda_1]$, the dimension has the bound
\[
\dimH R^*(\pi_{\lambda}(\bi), \lambda, \gamma)\geq\frac{2\log 2+\log \lambda_0}{\log (2/\gamma)}-\eta(\lambda_0,\lambda_1). 
\]

Let $\epsilon_n=\tfrac 1n$. Notice that the sets
\[
\{\lambda\in [\lambda_0, \lambda_1]\mid \dimH R^*(\pi_{\lambda}(\bi), \lambda, \gamma)\geq\frac{2\log 2+\log \lambda_0}{\log (2/\gamma)}-\eta(\lambda_0,\lambda_1)-C_0\epsilon_n\}
\]
are nested and by Lemma \ref{lowerboundmain} each have measure $\lambda_1-\lambda_0 - \epsilon_n$. Hence, by continuity of measures,  
\[
\mathcal L \{\lambda\in [\lambda_0, \lambda_1]\mid\dimH R^*(\pi_{\lambda}(\bi), \lambda, \gamma)\geq\frac{2\log 2+\log \lambda_0}{\log (2/\gamma)}-\eta(\lambda_0,\lambda_1)\}=\lambda_1 - \lambda_0. 
\]
In particular, for $\mathcal L$-almost all $\lambda\in [\lambda_0, \lambda_1]$
\[
\dimH R^*(\pi_{\lambda}(\bi), \lambda, \gamma)\geq \frac{2\log 2+\log \lambda}{\log (2/\gamma)}+\frac{\log (\lambda_0/\lambda_1)}{\log(2/\gamma)}-\eta(\lambda_0,\lambda_1).
\]
Since we can split the interval $[\tfrac 1{2\gamma},\bar{\lambda}]$ into finitely many intervals of however small length, and in particular, make $\log(\lambda_0/\lambda_1)$ and $\eta(\lambda_0,\lambda_1)$ uniformly as small as we like, we can obtain that 
for $\nu$ almost all $\bi\in\Sigma$ for $\mathcal L$-almost all $\lambda\in [\tfrac 1{2\gamma}, \overline\lambda)$
\[
\dimH R^*(\pi_{\lambda}(\bi), \lambda, \gamma)\geq \frac{2\log 2+\log \lambda}{\log (2/\gamma)}.
\]
By Fubini it follows that for $\mathcal L$-almost all $\lambda\in [\tfrac 1{2\gamma}, \overline\lambda)$, for $\nu$ almost all $\bi\in\Sigma$ 
\[
\dimH R^*(\pi_{\lambda}(\bi), \lambda, \gamma)\geq \frac{2\log 2+\log \lambda}{\log (2/\gamma)}
\]
and the result we want now follows as $\bar\nu_{\lambda}=\nu\circ\pi_{\lambda}^{-1}$.
\end{proof}

\bibliographystyle{plain}
\bibliography{vaitbib}

\begin{thebibliography}{10}

\bibitem{AllenBarany21}
Demi Allen and Bal\'azs B\'ar\'any.
\newblock On the {H}ausdorff measure of shrinking target sets on self-conformal
  sets.
\newblock {\em Mathematika}, 67(4):807--839, 2021.

\bibitem{Baker24}
Simon Baker.
\newblock Intrinsic {D}iophantine approximation for overlapping iterated
  function systems.
\newblock {\em Math. Ann.}, 388(3):3259--3297, 2024.

\bibitem{BakerKoivusalo24}
Simon Baker and Henna Koivusalo.
\newblock Quantitative recurrence and the shrinking target problem for
  overlapping iterated function systems.
\newblock {\em Adv. Math.}, 442:Paper No. 109538, 65, 2024.

\bibitem{BaranyHochmanRapaport19}
Bal\'azs B\'ar\'any, Michael Hochman, and Ariel Rapaport.
\newblock Hausdorff dimension of planar self-affine sets and measures.
\newblock {\em Invent. Math.}, 216(3):601--659, 2019.

\bibitem{BaranyKaenmaki17}
Bal\'azs B\'ar\'any and Antti K\"aenm\"aki.
\newblock Ledrappier-{Y}oung formula and exact dimensionality of self-affine
  measures.
\newblock {\em Adv. Math.}, 318:88--129, 2017.

\bibitem{BaranyRams18}
Bal\'azs B\'ar\'any and Micha\l\ Rams.
\newblock Shrinking targets on {B}edford-{M}c{M}ullen carpets.
\newblock {\em Proc. Lond. Math. Soc. (3)}, 117(5):951--995, 2018.

\bibitem{BaranySimonSolomyak23}
Bal\'azs B\'ar\'any, K\'aroly Simon, and Boris Solomyak.
\newblock {\em Self-similar and self-affine sets and measures}, volume 276 of
  {\em Mathematical Surveys and Monographs}.
\newblock American Mathematical Society, Providence, RI, [2023] \copyright
  2023.

\bibitem{BaranyTroscheit22}
Bal\'azs B\'ar\'any and Sascha Troscheit.
\newblock Dynamically defined subsets of generic self-affine sets.
\newblock {\em Nonlinearity}, 35(10):4986--5013, 2022.

\bibitem{BugeaudHarrapKristenseVelani10}
Yann Bugeaud, Stephen Harrap, Simon Kristensen, and Sanju Velani.
\newblock On shrinking targets for {$\Bbb Z^m$} actions on tori.
\newblock {\em Mathematika}, 56(2):193--202, 2010.

\bibitem{ChernovKleinbock01}
N.~Chernov and D.~Kleinbock.
\newblock Dynamical {B}orel-{C}antelli lemmas for {G}ibbs measures.
\newblock {\em Israel J. Math.}, 122:1--27, 2001.

\bibitem{Edgar98}
Gerald~A. Edgar.
\newblock {\em Integral, probability, and fractal measures}.
\newblock Springer-Verlag, New York, 1998.

\bibitem{Erdos39}
Paul Erd{\H{o}}s.
\newblock On a family of symmetric {B}ernoulli convolutions.
\newblock {\em Amer. J. Math.}, 61:974--976, 1939.

\bibitem{Galatolo10}
Stefano Galatolo.
\newblock Hitting time in regular sets and logarithm law for rapidly mixing
  dynamical systems.
\newblock {\em Proc. Amer. Math. Soc.}, 138(7):2477--2487, 2010.

\bibitem{HillVelani95}
Richard Hill and Sanju~L. Velani.
\newblock The ergodic theory of shrinking targets.
\newblock {\em Invent. Math.}, 119(1):175--198, 1995.

\bibitem{HillVelani99}
Richard Hill and Sanju~L. Velani.
\newblock The shrinking target problem for matrix transformations of tori.
\newblock {\em J. London Math. Soc. (2)}, 60(2):381--398, 1999.

\bibitem{Hochman14}
Michael Hochman.
\newblock On self-similar sets with overlaps and inverse theorems for entropy.
\newblock {\em Ann. of Math. (2)}, 180(2):773--822, 2014.

\bibitem{Hutchinson81}
John Hutchinson.
\newblock Fractals and self-similarity.
\newblock {\em Indiana Univ. Math. J.}, 30(5):713--747, 1981.

\bibitem{KershnerWintner35}
Richard Kershner and Aurel Wintner.
\newblock On {S}ymmetric {B}ernoulli {C}onvolutions.
\newblock {\em Amer. J. Math.}, 57(3):541--548, 1935.

\bibitem{KoivusaloRamirez18}
Henna Koivusalo and Felipe~A. Ram\'irez.
\newblock Recurrence to shrinking targets on typical self-affine fractals.
\newblock {\em Proc. Edinb. Math. Soc. (2)}, 61(2):387--400, 2018.

\bibitem{LiWangWuXu14}
Bing Li, Bao-Wei Wang, Jun Wu, and Jian Xu.
\newblock The shrinking target problem in the dynamical system of continued
  fractions.
\newblock {\em Proc. Lond. Math. Soc. (3)}, 108(1):159--186, 2014.

\bibitem{PeresSolomyak98}
Yuval Peres and Boris Solomyak.
\newblock Self-similar measures and intersections of {C}antor sets.
\newblock {\em Trans. Amer. Math. Soc.}, 350(10):4065--4087, 1998.

\bibitem{PrzytyckiUrbanski89}
F.~Przytycki and M.~Urba{\'n}ski.
\newblock On the {H}ausdorff dimension of some fractal sets.
\newblock {\em Studia Math.}, 93(2):155--186, 1989.

\bibitem{ShenWang13}
LuMing Shen and BaoWei Wang.
\newblock Shrinking target problems for beta-dynamical system.
\newblock {\em Sci. China Math.}, 56(1):91--104, 2013.

\bibitem{Shmerkin}
Pablo Shmerkin.
\newblock On {F}urstenberg's intersection conjecture, self-similar measures,
  and the {$L^q$} norms of convolutions.
\newblock {\em Ann. of Math. (2)}, 189(2):319--391, 2019.

\bibitem{ShmerkinSolomyak06}
Pablo Shmerkin and Boris Solomyak.
\newblock Zeros of {$\{-1,0,1\}$} power series and connectedness loci for
  self-affine sets.
\newblock {\em Experiment. Math.}, 15(4):499--511, 2006.

\bibitem{Solomyak95}
Boris Solomyak.
\newblock On the random series {$\sum\pm\lambda^n$} (an {E}rd{\H o}s problem).
\newblock {\em Ann. of Math. (2)}, 142(3):611--625, 1995.

\bibitem{Tseng08}
Jimmy Tseng.
\newblock On circle rotations and the shrinking target properties.
\newblock {\em Discrete Contin. Dyn. Syst.}, 20(4):1111--1122, 2008.

\bibitem{Varju2019}
P\'eter~P. Varj\'u.
\newblock On the dimension of {B}ernoulli convolutions for all transcendental
  parameters.
\newblock {\em Ann. of Math. (2)}, 189(3):1001--1011, 2019.

\end{thebibliography}

\end{document}